\documentclass[a4,11pt]{article}
\usepackage{amsmath,amsthm,amssymb}
\usepackage{mathrsfs}
\usepackage{braket}
\usepackage{comment}
\usepackage[top=25truemm,bottom=25truemm,left=25truemm,right=25truemm]{geometry}
\numberwithin{equation}{section}
\allowdisplaybreaks

\allowdisplaybreaks
\newtheorem{definition}{Definition}[section]
\newtheorem*{definition*}{Definition}
\newtheorem{proposition}[definition]{Proposition}

\newtheorem{theorem}[definition]{Theorem}
\newtheorem*{theorem*}{Theorem}
\newtheorem{remark}[definition]{Remark}
\newtheorem{corollary}[definition]{Corollary}
\newtheorem{lemma}[definition]{Lemma}

\newtheorem{fact}[definition]{Fact}

\newcommand{\di}{\displaystyle}

\newcommand{\fgl}{\mathfrak{gl}}

\newcommand{\bof}{\alpha}%first boson
\newcommand{\chf}{\mathcal{Q}}
\newcommand{\zef}{\mathsf{a}_0}%first zero mode
\newcommand{\zchf}{\mathsf{Q}}
\newcommand{\boc}{\bar{h}}%Cartan mode
\newcommand{\chc}{\bar{Q}}
\newcommand{\zec}{\bar{\mathsf{a}}_0}
\newcommand{\zchc}{\bar{\mathsf{Q}}}
\newcommand{\Lamc}{\bar{\Lambda}}
\newcommand{\boo}[1][]{h^{\perp #1}}% boson orthogonal to Cartan
\newcommand{\cho}{Q^{\perp}}
\newcommand{\zeo}{\mathsf{a}^{\perp}_0}
\newcommand{\zcho}{\mathsf{Q}^{\perp}}
\newcommand{\Lamo}{\Lambda^{\perp}}
\newcommand{\bon}{a}%normal boson before q-deformation
\newcommand{\chn}{Q}
\newcommand{\bons}{\widetilde{a}}%second normal boson before q-deformation
\newcommand{\chns}{\widetilde{Q}}

\usepackage{xspace}
\newcommand{\svirm}{\mathsf{SVir}}
\newcommand{\svir}{$\svirm$\xspace}

\newcommand{\NPb}{\mathop{{\rlap{\raise.65ex\hbox{$\scriptscriptstyle\bullet$}}
\lower.3ex\hbox{$\scriptscriptstyle\bullet$}}}}
\newcommand{\NPc}{{\rlap{\raise.4ex\hbox{$\scriptscriptstyle\circ$}}
\lower.4ex\hbox{$\scriptscriptstyle\circ$}}}
\title{Direct Sum Structure of the Super Virasoro Algebra and a Fermion Algebra Arising from the Quantum Toroidal $\fgl_2$}
\author{Yusuke~Ohkubo}
\date{}
\AtEndDocument{\bigskip {\small 
Institute for Fundamental Sciences, Setsunan University \par 
17-8 Ikeda Nakamachi, Neyagawa, Osaka 572-8508, Japan \par
\textit{E-mail}: \texttt{yusuke.ohkubo.math@gmail.com}
}}
\begin{document}

\maketitle

\begin{abstract}
It is known that the $q$-deformed Virasoro algebra can be constructed from 
a certain representation of the quantum toroidal $\fgl_1$ algebra. 
In this paper, we apply the same construction to the quantum toroidal algebra of type 
$\fgl_2$ and study the properties of resulting generators $W_i(z)$ ($i=1,2$). 
The algebra generated by $W_i(z)$ can be regarded as a 
$q$-deformation of the direct sum $\mathsf{F} \oplus \svirm$, 
where $\mathsf{F}$ denotes the free fermion algebra and \svir stands for the $N=1$ super Virasoro algebra, 
also referred to as the $N=1$ superconformal algebra or the Neveu-Schwarz-Ramond algebra. 
Moreover, the generators $W_i(z)$ admit two screening currents, and we show that 
their degeneration limits coincide with the screening currents of \svir. 
We also establish quadratic relations satisfied by $W_i(z)$
and show that they generate a pair of commuting $q$-deformed Virasoro algebras, 
which degenerate into two nontrivial commuting Virasoro algebras included in $\mathsf{F} \oplus \svirm$.  
\end{abstract}

\section{Introduction}
The quantum toroidal $\fgl_1$ algebra (hereafter denoted by $\mathcal{E}_1$)
or the Ding--Iohara--Miki algebra \cite{DI1997Generalization,Miki:2007}
possesses a free field realization associated with the Macdonald polynomials \cite{FHHSY2009commutative}. 
By extending this realization to the $N$-fold tensor product of Fock spaces, 
we can obtain the so-called level $N$ representation. 
From this level $N$ representation, 
by decoupling a Heisenberg algebra corresponding to the Cartan subalgebra, 
we can obtain the $q$-deformed Virasoro algebra or more generally the $q$-deformed W-algebra 
$\mathsf{W}_{q,t}(\mathfrak{sl}_N)$ \cite{SKAO1995quantum, AKOS1995Quantum, FF1995quantum}, 
as demonstrated in \cite{FHSSY2010Kernel}. 
Thus, the level 
$N$ representation of $\mathcal{E}_1$ can be viewed 
as the algebra $\mathsf{H} \otimes \mathsf{W}_{q,t}(\mathfrak{sl}_N)$. 
Here, we denote by $\mathsf{H}$ the $U(1)$ Heisenberg algebra. 
Furthermore, it was conjectured in \cite{AFHKSY2011notes} that in this representation, 
$q$-deformations of the conformal blocks in two-dimensional conformal field theories 
coincide with the Nekrasov partition functions of five-dimensional (K-theoretic) $U(N)$ gauge theories. 
This is one of five-dimensional analogs of the so-called AGT correspondence \cite{AGT2010liouville}, 
and the conjecture has subsequently been proved, including a formula for the Kac determinant in the level 
$N$ representation, in \cite{Ohkubo2017Kac,FOS2020Generalized}. 
For an interpretation in the context of geometric representation theory, see \cite{Negut2016qAGTW}. 
The five-dimensional AGT correspondence based on the level $N$ representation can be regarded as a 
$q$-analogue of the work of Alba, Fateev, Litvinov, and Tarnopolsky \cite{AFLT2011combinatorial,FL2011Integrable}. 
In their approach, they considered a special basis on the modules of the algebras 
$\mathsf{H}\oplus \mathsf{Vir}$ or $\mathsf{H}\oplus \mathsf{W}_N$
and provided a natural interpretation of the correspondence with four-dimensional $U(N)$ gauge theories. 
Here, $\mathsf{Vir}$ and $\mathsf{W}_N$ denote the usual Virasoro and $W_N$-algebras, respectively.

In this paper, we apply the same construction used to obtain the $q$-deformed Virasoro algebra from $\mathcal{E}_1$ 
to the quantum toroidal algebra of type $\fgl_2$ (hereafter denoted by $\mathcal{E}_2$), 
and study the resulting algebraic structures. 
The obtained generators are $W_i(z)$ ($i=1,2$),  defined as follows. 
The precise definition of $W_i(z)$ is given in Definition \ref{def: W}. 
These generators are the main objects of study in this paper. 
We also mention that preliminary computations of this work were reported in the bulletin \cite{Ohkubo2025toward}.

\begin{definition*}%\label{def: W}
Set 
\begin{align}
W_i(z)&=\Lambda^{+}_i(z)+\Lambda^{-}_i(z)\qquad (i=1,2), \\
\Lambda^{+}_i(z)&
=:\exp\left( -\sum_{r\neq 0}\frac{q^n}{n} \boo_{i,n}z^{-n}\right):
e^{\cho_i}(q^{-1}z)^{\boo_{i,0}}q_1^{\zeo},\\
\Lambda^{-}_i(z)&
=:\exp\left( \sum_{n\neq 0}\frac{q^{-n}}{n} \boo_{i,n}z^{-n}\right):
e^{-\cho_i}(qz)^{-\boo_{i,0}}q_1^{-\zeo}. 
\end{align}
Here, we used the Heisenberg algebra generated by $\boo_{i,n}, \cho$ and $\zeo, \zcho$ ($n\in\mathbb{Z},\; i=1,2$)
with the commutation relations 
\begin{align}
&[\boo_{i,n},\boo_{j,m}]=
\begin{cases}
n\delta_{n+m,0}, & i=j,  \\ 
-n \dfrac{q_1^n+q_3^n}{1+q_2^{-n}} \delta_{n+m,0}, &i \neq j, 
\end{cases} 
\quad [\boo_{i,0}, \cho_j]=\begin{cases}
1 & i=j, \\ -1& i\neq j, 
\end{cases}\quad 
[\zeo, \zcho]=\frac{\beta}{2}.
\end{align}
together with the conditions $\boo_{1,0}=-\boo_{2,0}, \cho_1=-\cho_2$. 
The other commutation relations are zero. 
As for the parameters $q, q_1,q_2,q_3$ and $\beta$, 
see Section \ref{sec: main op}. 
\end{definition*}

The algebra generated by $W_i(z)$ can be regarded as a 
$q$-deformation of the direct sum $\mathsf{F}\oplus \svirm$, 
where 
$\mathsf{F}$ denotes the free fermion algebra and \svir stands for the $N=1$ super Virasoro algebra 
(also called the $N=1$ superconformal algebra or the Neveu--Schwarz--Ramond algebra). 
Although $W_i(z)$ is written in terms of the two bosons, 
in the degenerate limit
one of them can be reinterpreted as a pair of fermions via the boson-fermion correspondence. 
This reinterpretation establishes an explicit connection with the free field realization of $\mathsf{F}\oplus \svirm$. 

The appearance of such an algebra can be understood in the context of the AGT correspondence.
In the undeformed setting, 
the gauge theories on the ALE space $ALE_n$ of type $A_n$
(a resolution of the orbifold $\mathbb{C}^2/{\mathbb{Z}_n}$)
have been related to superconformal field theories with the symmetry algebra \svir 
or its generalizations.  For example, see 
\cite{BF2011Super,BBFLT2011Instanton,BBT2012Bases,Ito2011Ramond,NT2011Central,IOY20132d4d,IOY2014qVirasoro}.
In particular, the work in \cite{BBFLT2011Instanton} 
extends the approach based on the special basis in \cite{AFLT2011combinatorial,FL2011Integrable} 
to the superconformal field theory with the symmetry algebra 
$\mathsf{H}\oplus\mathsf{H}\oplus\mathsf{F}\oplus\svirm$, 
which corresponds to the $U(2)$ gauge theory on $ALE_2$.%
\footnote{As a more general framework, a coset conformal field theory corresponding to $U(r)$ gauge theories on $\mathbb{C}^2/\mathbb{Z}_n$ has been proposed in \cite{BF2011Super}.}
In the module of this algebra,  
a special basis was constructed, 
whose matrix elements of the primary field reproduce the Nekrasov factors. 
On the other hand, in the $q$-deformed setting, 
there is another approach based on the use of trivalent intertwiners of the quantum toroidal algebras. 
 In \cite{AFS2012quantum}, trivalent intertwiners for $\mathcal{E}_1$ were introduced, 
reproducing Awata--Kanno's and Iqbal--Kozkaz--Vafa's refined topological vertexes \cite{AK2008Refined,IKV2007Refined}. 
The suitable compositions of these intertwiners yield the Nekrasov partition functions of five-dimensional 
gauge theories on $\mathbb{R}^4\times S^1$. 
This construction has been generalized to the quantum toroidal algebras of type $\fgl_n$ ($n\geq 3$), 
in \cite{AKMMSZ2018KZ}, 
which reproduce the Nekrasov partition functions on $ALE_n \times S^1$. 
In light of these developments, 
it is natural to expect that, 
starting from a suitable representation of $\mathcal{E}_2$,  
we can obtain a $q$-deformation of $\mathsf{F}\oplus \svirm$ 
by decoupling two Heisenberg algebras in a similar manner to the $\fgl_1$ case. 
A more ambitious goal is to decouple a component corresponding to the fermion algebra $\mathsf{F}$ 
and construct a $q$-deformation of the pure super Virasoro algebra \svir. 
At present, however, an efficient method for removing the contribution of $\mathsf{F}$
from the generators $W_i(z)$ remains elusive. 
We also note that a $q$-deformation of the $N=2$ superconformal algebra 
was recently proposed in \cite{AHKS2024quantum}.

The algebra generated by $W_i(z)$ has two screening currents $S^{\pm}(z)$ (See Definition \ref{def: screening}), 
which are essentially the same as those given in \cite{Zenkevich2019gln}. 
Moreover, in \cite{Zenkevich2019gln} screening currents associated with the quantum toroidal $\fgl_n$ algebra for general $n$ were also constructed. 
Each screening current $S^{\pm}(z)$ is expressed as a sum of two exponential terms, 
and exhibits a structure similar to that of the bosonic screening for the quantum affine algebra $U_q(\widehat{\mathfrak{sl}}_2)$ \cite{Matsuo1994qdeformation} or the $q$-deformed $N=2$ superconformal algebra \cite{AHKS2024quantum}. 
However, a comparison of the operator product formulas of $S^{\pm}(z)$ shows differences 
(See Appendix \ref{sec: OPE screening}). 
In the degenerate limit, the screening currents $S^{\pm}(z)$ reduce to 
those of \svir \cite{KM1986null} via the boson-fermion correspondence.

It is known that the singular vectors of \svir correspond to the Uglov polynomials
\cite{BBT2012Bases,BV2022NSR,Yanagida2015singular}. 
The singular vectors obtained from the screening currents $S^{\pm}(z)$ 
are expected to correspond to an uplift of the Uglov polynomials. 
For related results on the correspondence between Jack polynomials and the Virasoro or 
W-algebras, 
and on that between Macdonald polynomials and the $q$-deformed Virasoro or W-algebras, 
see \cite{MY1995Singular,MY1995RIMSkokyuroku,AMOS1995Excited,SKAO1995quantum,AKOS1995Quantum}.

Moreover, the relations of $W_i(z)$ allow us to generate a family of operators 
$\mathcal{T}(\xi;z)$ depending on a non-zero complex parameter $\xi$ 
(See Definition \ref{def: T op.} and the shorthand notation (\ref{eq: T shorthand})). 
For the special choices $(\xi_1, \xi_2)=(q_1^{\pm 1}, q_3^{\pm 1})$, 
two operators $\mathcal{T}(\xi_1,z)$ and $\mathcal{T}(\xi_2,w)$ commute, 
and they satisfy the relation of the $q$-deformed Virasoro algebra. 
The resulting relations can be summarized as follows (See Section \ref{sec: quad rel}). 
\begin{theorem*}We obtain 
\begin{gather}
W_i(z)W_i(w)+W_i(w)W_i(z)
=qz^{-1}\delta\left( \frac{q_2w}{z} \right) +q^{-1}z^{-1}\delta\Big( \frac{w}{q_2z} \Big),\\
\xi \cdot f\left(\xi; \frac{w}{z}\right) W_1(z) W_2(w) 
+f\left(\xi^{-1} ; \frac{z}{w} \right) W_2(w) W_1(z)
= \delta \Big( \frac{\xi w}{z}\Big) w^{-1} \mathcal{T}(\xi; w),
\end{gather}
\begin{gather}
\big[\mathcal{T}(q_1;z), \mathcal{T}(q_3;w)\big]
=\left[\mathcal{T}(q_1^{-1};z), \mathcal{T}(q_3^{-1};w)\right]=0,\\
g^{(k)}\left( \frac{w}{z} \right) \mathcal{T}(q_k^{\pm 1};z)\mathcal{T}(q_k^{\pm 1};w)
-g^{(k)}\left( \frac{z}{w} \right) \mathcal{T}(q_k^{\pm 1};w)\mathcal{T}(q_k^{\pm 1};z)
=\mathcal{C}^{(k)} \cdot \Big(  \delta \Big( \frac{w}{q_2z}\Big) - \delta \left( \frac{q_2w}{ z}\right) \Big), 
\end{gather}
\begin{gather}
q_k^{\pm 1} \cdot f\left( q_k^{\pm 1} ; \frac{w}{z} \right) W_1(z) \mathcal{T}(q_k^{\pm 1};w) 
- f\left( q_k^{\mp 1};\frac{z}{w} \right) \mathcal{T}(q_k^{\pm 1} ;w)  W_1(z)\nonumber \\
=\pm q^{\pm 1} (q_k-q_k^{-1}) \delta \bigg( \frac{q_k^{\pm 1} w}{q_2^{\pm 1}z} \bigg) W_2(w),\\
q_k^{\mp 1} \cdot f\Big( q_k^{\mp 1} ; \frac{q_k^{\pm 1} w}{z} \Big) W_2(z) \mathcal{T}(q_k^{\pm 1};w) 
- f\Big( q_k^{\pm 1};\frac{z}{q_k^{\pm 1} w} \Big) \mathcal{T}(q_k^{\pm 1};w)  W_2(z)\nonumber\\
=\mp q^{\mp 1} (q_k-q_k^{-1})\delta \left( \frac{q_2^{\pm 1}w}{z} \right)W_1(q_k^{\pm 1} w),
\end{gather}
\begin{comment}
\begin{gather} % only T(q_k;z)
q_k \cdot f\left( q_k ; \frac{w}{z} \right) W_1(z) \mathcal{T}(q_k;w) 
- f\left( q_k^{-1};\frac{z}{w} \right) \mathcal{T}(q_k;w)  W_1(z)\nonumber \\
=q(q_k-q_k^{-1})\delta \left( \frac{q_kw}{q_2z} \right)W_2(w),\\
q_k^{-1} \cdot f\left( q_k^{-1} ; \frac{q_kw}{z} \right) W_2(z) \mathcal{T}(q_k;w) 
- f\Big( q_k;\frac{z}{q_kw} \Big) \mathcal{T}(q_k;w)  W_2(z)\nonumber\\
=-q^{-1} (q_k-q_k^{-1})\delta \left( \frac{q_2w}{z} \right)W_1(q_k w).
\end{gather}
\end{comment}
for $i=1,2$ and $k=1,3$. 
Here, we set
\begin{gather}
f(\xi; z)
%=\frac{1}{1-\xi z}f\left(z\right)
=\exp\left\{ \sum_{n=1}^{\infty} \left( \xi^n -\frac{q_1^n+q_3^n}{(1+q_2^{-n})} \right) \frac{z^n}{n} \right\},\\
g^{(k)}\left( z\right)=\exp \left( \sum_{n>0} \frac{(1-q_k^{2n})(1-q_2^{-n}q_k^{-2n})}{n(1+q_2^{-n})} z^n \right), \quad 
\mathcal{C}^{(k)}=-\frac{(1-q_k^2)(1-q_2^{-1}q_k^{-2})}{1-q_2^{-1}}. 
\end{gather}
\end{theorem*}
In the degenerate limit,
the commuting operators $\mathcal{T}(q_1^{\pm 1};z)$ and $\mathcal{T}(q_3^{\pm 1};z)$ give rise to two nontrivial commuting Virasoro algebras included in $\mathsf{F}\oplus \svirm$. 
These Virasoro pairs serve as a main tool in the construction of the special basis in the work of \cite{BBFLT2011Instanton}. 
Therefore, the operators $\mathcal{T}(\xi;z)$ are expected to provide a natural $q$-deformation of that basis. 
We note that $q$-deformation of \cite{BBFLT2011Instanton} is essentially suggested in \cite{FJMM2016branching},
and the operators $\mathcal{T}(q_k^{\pm 1};z)$ ($k=1,3$) 
can be viewed as analogues of the fused currents introduced there. 
In addition, 
for $\xi =q_k^{\pm 1}$ ($k=1,3$), the relations among $W_i(z)$ and $\mathcal{T}(\xi;z)$ close, 
so that we can define an algebra with these as defining relations. 
These relations also make it possible to analyze highest weight modules. 
We leave such investigations for future work.

There have been a lot of related studies on derivation of various deformed W-algebras 
from quantum toroidal algebras, including supersymmetric cases. 
For example, see \cite{FJMV2021Deformations,FJM2018towards}. 
The related constructions also include 
a $q$-deformation of the corner vertex operator algebra (Gaiotto--Rapcak's $Y$-algebra \cite{GR2017Vertex})
discussed in \cite{Harada2020Quantum,HMNW2021deformation} and 
the $q$-deformation of the $N=2$ superconformal algebra mentioned above \cite{AHKS2024quantum}. 
Based on the structure of the screening currents, however, the generators 
$W_i(z)$ introduced in this paper appear to differ from those of the algebras. 
%It remains possible that they become equivalent under a suitable transformation, 
%or that some operators generated from $W_i(z)$ coincide with the generators of these algebras.
%Clarifying these connections is left for future work.

This paper is organized as follows. 
In Section \ref{sec: main op}, 
the free field realization of $\mathcal{E}_2$ is given, 
and the main operators $W_i(z)$ are derived. 
In Section \ref{sec: lim}, 
we provide a brief review of the free field realization and the screening currents of \svir. 
We further discuss the degenerate limit of the generators $W_i(z)$, 
from which the algebra $\mathsf{F} \oplus \svirm$ appears. 
These two sections revisit the results previously reported in \cite{Ohkubo2025toward}, 
with minor adjustments and improved exposition.
In Section \ref{sec: screening}, 
we introduce the screening currents $S^{\pm}(z)$
and show that they reduce to the ones of \svir. 
In Section \ref{sec: quad rel}, 
we calculate the quadratic relations satisfied by $W_i(z)$ and $\mathcal{T}(\xi;z)$, 
and we investigate the degenerate limit of $\mathcal{T}(\xi;z)$. 
We prove the free field realization of $\mathcal{E}_2$ in Appendix \ref{sec: pf rep} 
and prove some formulas on the boson--fermion correspondence in Appendix \ref{sec: fml bf corresp}. 
We present formulas for the operator products of the screening currents $S^{\pm}(z)$ in Appendix \ref{sec: OPE screening}.

\section{Derivation of the main operator}\label{sec: main op}

In this section, we describe the definition of the quantum toroidal $\fgl_2$ algebra $\mathcal{E}_2$ 
and present its free field realization.
We further decompose the Cartan modes from its tensor representation 
and derive the main operators $W_i(z)$.
The definition of $\mathcal{E}_2$ follows \cite{FJMM2016branching}.

\subsection{Definition of  the quantum toroidal $\fgl_2$ algebra}
Let $q$ and $d$ be complex parameters satisfying that $q^n d^m \neq 1$ for any $n,m\in \mathbb{Z}$ ($n\neq 0$ or $m \neq 0$). 
Set 
\begin{equation}
q_1=q^{-1}d, \quad 
q_2=q^2, \quad 
q_3=q^{-1}d^{-1}. 
\end{equation}
Note that $q_1q_2q_3=1$. Further, we set 
\begin{align}
\beta = - \frac{\log q_3}{\log q_1}
\end{align}
so that $q_3=q_1^{-\beta}$. 

$\mathcal{E}_2$ is the unital associative algebra generated by 
$E_{i,n}, F_{i,n}, H_{i,k}, K_i^{\pm 1} (n \in \mathbb{Z},\; k \in \mathbb{Z}_{\neq 0}, \;  i=1,2)$
and central elements $q^{\pm c}$. We use the following formal generating series: 
\begin{align}
&E_i(z) =\sum_{n\in \mathbb{Z}}E_{i,n}z^{-n}, \quad 
F_i(z) =\sum_{n\in \mathbb{Z}}F_{i,n}z^{-n}, \quad\\ 
%K_i^{\pm}(z) =\sum_{\pm k\in \mathbb{Z}_{\geq 0}}K_{i,k}z^{-k}
&K_i^{ \pm}(z)=K_i^{ \pm 1} \exp \left( \pm\left(q-q^{-1}\right) \sum_{n=1}^{\infty} H_{i, \pm n} z^{\mp n}\right). 
\end{align}
The defining relations are given by 
\begin{align}
K_i K_i^{-1}=K_i^{-1} K_i=1,& \qquad q^c q^{-c}=q^{-c} q^c=1, \label{eq: def rel KK c}\\
K^\pm_i(z)K^\pm_j (w) &= K^\pm_j(w)K^\pm_i (z),\label{eq: def rel KK} \\
\frac{g_{i,j}(q^{-c}z,w)}{g_{i,j}(q^cz,w)}K^-_i(z)K^+_j (w) 
&=\frac{g_{j,i}(w,q^{-c}z)}{g_{j,i}(w,q^cz)} K^+_j(w)K^-_i (z), \label{eq: def rel KK2}\\
d_{i,j}g_{i,j}(z,w)K_i^\pm(q^{(1\mp1)c/2}z)E_j(w)&+g_{j,i}(w,z)E_j(w)K_i^\pm(q^{(1\mp1) c/2}z)=0, \\
d_{j,i}g_{j,i}(w,z)K_i^\pm(q^{(1\pm1)c/2}z)F_j(w)&+g_{i,j}(z,w)F_j(w)K_i^\pm(q^{(1\pm1) c/2}z)=0\,,\\
d_{i,j}g_{i,j}(z,w)E_i(z)E_j(w)&+g_{j,i}(w,z)E_j(w)E_i(z)=0, \label{eq: def rel EE} \\
d_{j,i}g_{j,i}(w,z)F_i(z)F_j(w)&+g_{i,j}(z,w)F_j(w)F_i(z)=0, \label{eq: def rel FF}\\
[E_i(z),F_j(w)]=\frac{\delta_{i,j}}{q-q^{-1}}
&(\delta\bigl(q^c\frac{w}{z}\bigr)K_i^+(z)
-\delta\bigl(q^c\frac{z}{w}\bigr)K_i^-(w)),\label{eq: def rel EF}
\end{align}
\begin{align}
\underset{z_1, z_2, z_3}{\operatorname{Sym}}\Big[E_i (z_1),\Big[E_i (z_2), &\big[E_i (z_3), 
E_j(w)\big]_{q_2}\Big] \Big]_{q_2^{-1}}=0\quad (i\neq j), \label{eq:serr1} \\
\underset{z_1, z_2, z_3}{\operatorname{Sym}}\Big[F_i (z_1 ),\Big[F_i (z_2),&\big[F_i (z_3), 
F_j(w)\big]_{q_2}\Big]\Big]_{q_2^{-1}}=0\quad (i\neq j).\label{eq:serr2} 
\end{align}
Here, we have set
\begin{align}
&g_{i,j}(z,w)=\begin{cases}
	      z-q_2w & (i= j),\\
              (z-q_1w)(z-q_3w)& (i\neq  j),
	     \end{cases}\qquad 
d_{i,j}=
\begin{cases}
1& (i= j), \\
-1 & (i\neq j). \\
\end{cases}
\end{align}
$\delta_{i,j}$ is the Kronecker delta, 
and we used the formal delta function $\di \delta(z)=\sum_{n\in \mathbb{Z}} z^n$. 
In the Serre relations (\ref{eq:serr1}) and (\ref{eq:serr2}), 
$\underset{z_1, z_2, z_3}{\operatorname{Sym}}$ denotes the symmetrization with respect to $z_1, z_2, z_3$, 
and we put $[A, B]_p=AB-p BA$. 

Moreover, $\mathcal{E}_2$ is equipped with the coproduct structure. 
The formulas for the coproduct $\Delta$ are given by  
\begin{align}
&\Delta (E_i(z))=E_i(z)\otimes 1+K_i^-(C_1 z)\otimes E_i(C_1 z)\,, \\ %\quad 
&\Delta (F_i(z))=F_i(C_2 z)\otimes K_i^+(C_2 z)+1\otimes F_i(z)\,,\\
&\Delta (K^{+}_i(z))=K_i^+(z)\otimes K^+_i(C_1^{-1} z)\,, \quad 
\Delta (K^{-}_i(z))=K_i^-(C_2^{-1} z)\otimes K^-_i(z)\,,\\
&\Delta (q^{c})=q^{c}\otimes q^{c}\,. 
\end{align}
Here, we set $C_1=q^{c}\otimes 1$, $C_2=1\otimes q^{c}$. 

\begin{remark}
The generators $H_{i,n}$ form the Cartan subalgebra. The coproduct for $H_{i,n}$ takes the form 
\begin{align}\label{eq: Delta h}
\Delta (H_{i,-n})=C_2^{-n} H_{i,-n} \otimes 1 +1\otimes H_{i,-n}, \quad 
\Delta (H_{i,n})=H_{i,n} \otimes 1 +C_1^{n} 1\otimes H_{i,n} \quad (n>0).
\end{align}
\end{remark}

\subsection{Free field realization and decoupling of the Cartan part}

We define the Heisenberg algebra $\mathcal{H}_q$ generated by $\bof_{i,n}$, $\chf_i$ ($n\in \mathbb{Z}, i=1,2$)
and $\zef, \zchf$ with the commutation relations
\begin{align}
&[\bof_{i,n},\bof_{j,m}]=
\begin{cases}
n(1+q_2^{-|n|}) \delta_{n+m,0}, \quad i=j, \\
-n(q_1^{|n|}+q_3^{|n|})\delta_{n+m,0}, \quad i\neq j,
\end{cases}\\
&[\bof_{i, n}, \chf_j]=
\begin{cases}
2 \, \delta_{n,0}, \quad i=j, \\
-2 \, \delta_{n,0}, \quad i\neq j, 
\end{cases}
\quad [\zef, \zchf]=\beta, \quad 
\end{align}
together with the condition $\bof_{1,0}=-\bof_{2,0}, \chf_1=-\chf_2$. 
Suppose that the other commutation relations are zero. 
Let $\ket{0}$ be the highest weight vector satisfying $\bof_{i,n}\ket{0}=\zef \ket{0}=0$ ($n \geq 0$). 
For an integer $n$ and a complex number $u$, we define $\ket{n, u}=e^{\frac{n}{2}\chf_1+\frac{u}{\beta} \zchf}\ket{0}$, 
so that 
\begin{align}
\bof_{1,0} \ket{n,u} =n \ket{n,u}, \quad \bof_{2,0} \ket{n,u} = -n \ket{n,u}, \quad \zef \ket{n,u}=u\ket{n,u}.
\end{align} 
We define the Fock module $\mathcal{F}(n,u)$ by 
%$\mathcal{F}(n,u)=\mathcal{H}^-\cdot \ket{n,u}$ 
\begin{align}
\mathcal{F}(n,u)=\mathbb{C}[\bof_{1,-1}, \bof_{1,-2},\ldots, \bof_{2,-1},\bof_{2,-2},\ldots] \ket{n,u}
\end{align}
and set 
$\di \mathcal{F}_u=\bigoplus_{n \in \mathbb{Z}} \mathcal{F}(n,u)$. 
The algebra $\mathcal{E}_2$ admits a representation realized by the vertex operators introduced below. 
Similar representations were given in \cite{Saito1998quantum,STU1998toroidal}
for the quantum toroidal $\fgl_n$ algebra with $n\geq 3$. 
The representation presented in this paper 
is an adaptation of that construction to the case of $\fgl_2$, 
with a slight modification in the treatment of the zero modes. 
For the case of type $\fgl_1$, see \cite{FHHSY2009commutative}. 
See also \cite{FJM2018towards} for another free field realization of $\mathcal{E}_2$. 

\begin{definition}
Define the vertex operators
$\eta_{i}(z), \xi_{i}(z), \varphi_i^{\pm}(z)\in \mathrm{End}(\mathcal{F}_u)[[z, z^{-1}]]$ ($i=1,2$) by 
\begin{align}
&\eta_{i}(z)=
%\exp\left( \sum_{r>0}\frac{1}{r}\alpha_{i,-r}z^r \right)
%\exp\left( -\sum_{r>0}\frac{1}{r}\alpha_{i,r}z^{-r} \right),
:\exp\left( -\sum_{n\in \mathbb{Z}_{\neq 0}}\frac{1}{n}\alpha_{i,n}z^{-n} \right):, \qquad \qquad 
\xi_{i}(z)=
%\exp\left( -\sum_{r>0}\frac{q^r}{r}\alpha_{i,-r}z^r \right)
%\exp\left(\sum_{r>0} \frac{q^r}{r}\alpha_{i,r}z^{-r} \right),\\ 
:\exp\left(\sum_{n\in \mathbb{Z}_{\neq 0}} \frac{q^{|n|}}{n}\alpha_{i,n}z^{-n} \right):,\\
&\varphi_i^+(z)=\exp\left( -\sum_{n>0}\frac{1-q_2^{n}}{n}q^{-n/2}\alpha_{i,n}z^{-n} \right),\quad 
\varphi_i^-(z)=\exp\left( \sum_{n>0}\frac{1-q_2^{n}}{n}q^{-n/2}\alpha_{i,-n}z^n \right). 
\end{align}
Here, the symbol $:\ast:$ stands for the normal ordering of the Heisenberg algebra $\mathcal{H}_q$. 
\end{definition}

\begin{proposition}\label{prop: rep tor gl2}
The following assignment $\rho_u$ gives a representation of $\mathcal{E}_2$ on the Fock module $\mathcal{F}_u$: 
\begin{align}
&\rho_u (E_i(z))=\eta_i(z)\times e^{\chf_i}z^{\bof_{i,0}+1} q_1^{\zef},\quad 
\rho_u (F_i(z))=\xi_i(z)\times e^{-\chf_i}z^{-\bof_{i,0}+1}q_1^{-\zef},\label{eq: free field rep1}\\
&\rho_u(K_i^{+}(z))=\varphi^+_i(q^{-\frac{1}{2}}z)\times q^{\bof_{i,0}}, \quad 
\rho_u(K_i^{-}(z))=\varphi^-_i(q^{-\frac{1}{2}}z)\times q^{-\bof_{i,0}}, \label{eq: free field rep2}\\
&\rho_u(q^{\pm c})=q^{\pm 1}.
\end{align}
\end{proposition}

\begin{remark}
This representation holds even without the zero-mode condition $\bof_{1,0}=-\bof_{2,0}, \chf_1=-\chf_2$. 
This condition is imposed in order to ensure that \svir arises directly in the limit $q_i \rightarrow 1$ . 
\end{remark}

Although this representation is already known, 
for completeness we include the proof in Appendix \ref{sec: pf rep}. 
Using this free field realization, 
we consider the tensor representation of $E_i(z)$ on 
$\mathcal{F}_{u_1} \otimes \mathcal{F}_{u_2}$.
\begin{definition}
Set 
\begin{align}
X_i(z)= \rho_{u_1} \otimes \rho_{u_2} \circ \Delta (E_i(z)) \quad (i=1,2). 
\end{align}
\end{definition}

Explicitly, $X_i(z)$ is given by 
\begin{align}
&X_i(z)=\Lambda_{i,1}(z)+\Lambda_{i,2}(z), \\
&\Lambda_{i,1}(z) \equiv \left\{ \eta_i(z)\cdot e^{\chf_i}z^{\bof_{i,0}+1}q_1^{\zef}  \right\} \otimes 1,\\
&\Lambda_{i,2}(z) \equiv \left\{ \varphi_i(q^{1/2}z)\cdot q^{-\bof_{i,0}}  \right\}\otimes 
\left\{ \eta_i(q z) \cdot e^{\chf_i}(q z)^{\bof_{i,0}+1}q_1^{\zef} \right\}. 
\end{align}
We decompose the operator $X_i(z)$ 
into components corresponding to the Cartan subalgebra 
and components commuting with them.

\begin{definition}
%Define $\boc_{i,n}, \boo_{i,n} \in \widetilde{\mathcal{H}} \otimes \widetilde{\mathcal{H}}$ by 
Set
\begin{align}
&\boc_{i,n} = \rho_{u_1} \otimes \rho_{u_2} \circ \Delta (H_{i,n}),\\
&\boo_{i,n}=\bof_{i,n}\otimes 1
-\frac{[\bof_{i,n}\otimes 1,\boc_{i,-n}]}{[\boc_{i,n},\boc_{i,-n}]}\boc_{i,n}
\quad (n\neq 0).
\end{align}
Further we define the zero modes $\boc_{i,0},\,  \boo_{i,0},\,  \chc_i,\,  \cho_i$,
$\zec, \zeo$ and $\zchc ,\zcho$ by  
\begin{align}
&\boc_{i,0}=\frac{1}{2}(\bof_{i,0}\otimes 1+1\otimes\bof_{i,0}), \quad 
\boo_{i,0}=\frac{1}{2}(\bof_{i,0}\otimes 1-1\otimes\bof_{i,0}), \label{eq: decom of zero mode}\\
&\chc_i=\frac{1}{2}(\chf_i \otimes 1+1\otimes \chf_i), \quad  
\cho_i=\frac{1}{2}(\chf_i\otimes 1-1\otimes\chf_i), \\ 
&\zec=\frac{1}{2}(\zef \otimes 1+1\otimes \zef +1\otimes 1), \quad 
\zeo=\frac{1}{2}(\zef\otimes 1-1\otimes \zef -1\otimes 1),\\
&\zchc=\frac{1}{2}(\zchf \otimes 1+1\otimes \zchf),\quad 
\zcho=\frac{1}{2}(\zchf \otimes 1-1\otimes \zchf). 
\end{align}
\end{definition}

Note that $\boc_{1,0}=-\boc_{2,0},\;  \boo_{1,0}=-\boo_{2,0},  \;  \chc_1=-\chc_2,\; \cho_1 =-\cho_2$. 
The Cartan mode $\boc_{i,n}$ is of the form 
\begin{align}
\boc_{i,n}=-\frac{(1-q_2^n)}{n(q-q^{-1})}\left( \bof_{i,n}\otimes 1 +q^{|n|}\cdot 1\otimes \bof_{i,n} \right)\quad 
(n\neq 0). 
\end{align}
The component of $1\otimes \bof_{i,n}$ which commutes with the Cartan mode 
coincides with $\boo_{i,n}$ up to scalar multiplication. That is, we have
\begin{align}
-q^{-|n|}\boo_{i,n}=
1 \otimes \bof_{i,n}
-\frac{[1\otimes \bof_{i,n},\boc_{i,-n}]}{[\boc_{i,n},\boc_{i,-n}]}\boc_{i,n} \quad (n \in \mathbb{Z}_{\neq 0}). 
\end{align}
By direct calculations, we can obtain the following commutation relations. 

\begin{proposition}
It follows that 
\begin{align}
&[\boc_{i,n},\boc_{j,m}]=
\begin{cases}
\dfrac{q_2^{-2|n|}\left(1-q_2^{2|n|} \right)^2}{n (q-q^{-1})^2} \delta_{n+m,0}, & i=j,  \\ 
-\dfrac{q_2^{-|n|} \left(1-q_2^{|n|}\right)\left(1-q_2^{2|n|}\right) \left(q_1^{|n|}+q_3^{|n|} \right)}{n(q-q^{-1})^2} \delta_{n+m,0}, &i \neq j, 
\end{cases}\\ 
&[\boo_{i,n},\boo_{j,m}]=
\begin{cases}
n\delta_{n+m,0}, & i=j,  \\ 
-n \dfrac{q_1^n+q_3^n}{1+q_2^{-n}} \delta_{n+m,0}, &i \neq j, 
\end{cases}\qquad \quad 
[\boo_{i,n} , \boc_{j,m}]=0 \quad (\forall i, j), \\
&[\boo_{i,0}, \cho_j]=\begin{cases}
1 & i=j, \\ -1& i\neq j, 
\end{cases}\qquad \qquad 
%&z^{\boo_{i,0}}e^{\cho_i}=z e^{\cho_i}z^{\boo_{i,0}}, \quad 
%z^{\boo_{i,0}}e^{\cho_j}=z^{-1} e^{\cho_j}z^{\boo_{i,0}} \qquad (i\neq j), \\
[\zeo, \zcho]=\frac{\beta}{2}.
\end{align}
\end{proposition}

By setting 
\begin{align}
&\Lamo_{i,1}(z)=:\exp\left( -\sum_{n\neq 0} \frac{1}{n}\, \boo_{i,n}z^{-n}\right):
e^{\cho_i}z^{\boo_{i,0}}q_1^{\zeo }, \\
&\Lamo_{i,2}(z)=:\exp\left( \sum_{n\neq 0} \frac{q_2^{-n}}{n}\boo_{i,n}z^{-n}\right):
e^{-\cho_i}  z^{-\boo_{i,0}}q^{-2\boo_{i,0}}q_1^{-\zeo}, \\
&\Lamc_{i}(z)=
:\exp\left( -\sum_{n> 0}\frac{(q-q^{-1}) q_2^n}{(1-q_2^{2n})}\boc_{i,-n}z^{n}\right) 
\exp\left( \sum_{n>0}\frac{(q-q^{-1})}{(1-q_2^{2n})}\boc_{i,n}z^{-n}\right) :
e^{\chc_i}z^{\boc_{i,0}+1} q_1^{\zec},
\end{align}
we can decompose $X_i(z)$ into the Cartan part $\Lamc_{i}(z)$ and the component $\Lamo_{i,k}(z)$ 
which commutes with it. 
That is to say, we obtain  
\begin{align}
&X_i(z)=\left( \Lamo_{i,1}(z) + \Lamo_{i,2}(z)\right)\Lamc_i(z).
\end{align}

In the following, we decouple the Cartan part $\Lamc_{i}(z)$ 
and study the algebra generated by the $\Lamo_{i,k}(z)$.
In doing so, taking into account the symmetry, we employ normalized generators defined as follows.

\begin{definition}\label{def: W}
Set 
\begin{align}
W_i(z)&=\Lambda^{+}_i(z)+\Lambda^{-}_i(z)\qquad (i=1,2), \\
\Lambda^{+}_i(z)&
=:\exp\left( -\sum_{n\neq 0}\frac{q^n}{n} \boo_{i,n}z^{-n}\right):
e^{\cho_i}(q^{-1}z)^{\boo_{i,0}}q_1^{\zeo},\\
\Lambda^{-}_i(z)&
=:\exp\left( \sum_{n\neq 0}\frac{q^{-n}}{n} \boo_{i,n}z^{-n}\right):
e^{-\cho_i}(qz)^{-\boo_{i,0}}q_1^{-\zeo}. 
\end{align}
\end{definition}

These are the main generators of study in this paper. 
Note that we have 
\begin{align}
\Lamo_{i,1}(z) + \Lamo_{i,2}(z)= W_i(q z). 
\end{align}

\begin{remark}
Even if we start from the generator $F_i(z)$ and perform the same computation, 
the resulting operator is again precisely $W_i(z)$. That is, 
it follows that 
\begin{align}
\rho_{u_1}\otimes \rho_{u_2}\circ \Delta(F_i(z))=W_{i}(q z) \cdot  \Lamc^{\ast}_{i}(z). 
\end{align}
Here, we set 
\begin{align}
&\Lamc^{\ast}_{i}(z)=
:\exp\left( -\sum_{n> 0}\frac{(q-q^{-1})}{(1-q_2^{-2n})}\boc_{i,-n}z^{n}\right) 
\exp\left( \sum_{n>0}\frac{(q-q^{-1})q_2^{-n}}{(1-q_2^{-2n})}\boc_{i,n}z^{-n}\right) :
e^{-\chc_i}z^{-\boc_{i,0}+1} q_1^{-\zec+1}. 
\end{align}
\end{remark}

\section{Limit $q_i \rightarrow 1$}\label{sec: lim}

In this section, we show that the super Virasoro algebra \svir arises in the limit $q_i \to 1$ 
of the generator $W_i(z)$. 
To this end, 
we begin by reviewing the free field realization and the screening currents of \svir, 
following the formulation by Kato and Matsuda \cite{KM1986null}. 
While their realization employs one free bosonic field and one free fermionic field, 
our approach relies on the boson--fermion correspondence to reformulate the entire structure purely in terms of bosons.
Throughout the discussion, we restrict our attention to the Neveu--Schwarz sector.

\subsection{Super Virasoro algebra}

Consider the Heisenberg algebra $\mathcal{H}$ 
generated by $\bon_n, \chn$ and $\bons_n$, $\chns$ ($n \in \mathbb{Z}$), 
subject to the commutation relations 
\begin{align}
&[\bon_{n}, \bon_{m}]=[\bons_{n}, \bons_{m}]=n \delta_{n+m,0},\qquad 
[\bon_{n}, \bons_{m}]=0,  \label{eq: bon rel 1}\\
&[\bon_n , \chn]=[\bons_n, \chns]=\delta_{n,0}, \qquad [\bon_n,\chns]=[\bons_n,\chn]=0.\label{eq: bon rel 2}
\end{align}
Define the generating series (free bosonic fields) $\phi(z)$ and $\widetilde{\phi}(z)$ as
\begin{align}
&\phi(z) =\chn+\bon_0 \log z -\sum_{n\neq 0} \frac{1}{n} \bon_{n}z^{-n}, \\
&\widetilde{\phi}(z)=\chns+\bons_0 \log z -\sum_{n\neq 0} \frac{1}{n} \bons_{n}z^{-n}.
\end{align}
We apply the boson-fermion correspondence to the free bosonic field $\phi(z)$
and identify it with a pair of fermionic fields. 
Specifically, we define the generating series (free fermionic fields) 
\begin{align}
\psi(z)=\sum_{\mu \in \mathbb{Z}+\frac{1}{2}} \psi_{\mu} z^{-\mu-\frac{1}{2}}, \quad 
\widetilde{\psi}(z)=\sum_{\mu \in \mathbb{Z}+\frac{1}{2}} \widetilde{\psi}_{\mu} z^{-\mu-\frac{1}{2}}
\end{align}
via the correspondence 
\begin{align}\label{eq: realize psi}
\psi(z)=\frac{1}{\sqrt{-2}}\left(: e^{\phi(z)}: - :e^{-\phi(z)}: \right), \quad 
\widetilde{\psi}(z)=\frac{1}{\sqrt{2}}\left(: e^{\phi(z)}: + :e^{-\phi(z)}: \right).
\end{align}
Here and hereafter, we use the same normal ordering symbol $:\ast:$ for the Heisenberg algebra $\mathcal{H}$ 
as for $\mathcal{H}_q$. 
These fermions satisfy the canonical anticommutation relations:
\begin{align}
&[\psi_{\mu}, \psi_{\nu}]_{+}=[\widetilde{\psi}_{\mu}, \widetilde{\psi}_{\nu}]_{+}=\delta_{\mu+\nu,0}, \quad 
[\widetilde{\psi}_{\mu}, \psi_{\nu}]_{+}=0, 
\end{align}
where $[A, B]_+ =AB+BA$ denotes the anticommutator.

By using one bosonic and one fermionic field, 
we can construct a free field realization of \svir. 

\begin{definition}\label{def: SCA rep}
Let $\sigma$ be a complex parameter, and set 
\begin{align}
& T(z)=\frac{1}{2} :\widetilde{\phi}^{\prime}(z)^2:+\sigma \widetilde{\phi}^{\prime \prime}(z)
+\frac{1}{2} \NPb \psi^{\prime}(z) \psi(z) \NPb , \label{eq: rep of T(z)}\\
& G(z)=: \widetilde{\phi}^{\prime}(z): \psi(z)+2 \sigma \psi^{\prime}(z). \label{eq: rep of G(z)}
\end{align}
Here, $\NPb\ast\NPb$ is the normal ordering for the fermionic modes defined by 
\begin{align}
\NPb \psi_\mu \psi_{\nu} \NPb=
\begin{cases}
\psi_{\mu} \psi_{\nu}& \mu \geq \nu, \\
-\psi_{\nu} \psi_{\mu}& \mu < \nu. \\
\end{cases}
\end{align}
In $\widetilde{\phi}'(z)$ and $\widetilde{\psi}'(z)$, 
the prime symbol indicates differentiation with respect to $z$. 
Moreover, define the expansion coefficients $L_{n}$ and $G_{\mu}$ by 
\begin{align}
& T(z)=\sum_{n \in \mathbb{Z}} L_n z^{-n-2}, \quad  G(z)=\sum_{\mu \in \mathbb{Z}+\frac{1}{2}} G_{\mu} z^{-\mu-\frac{3}{2}}.
\label{eq: generating fn of T G}
\end{align}
\end{definition}

\begin{fact}
The operators 
$G_{\mu}$ and $L_n$ ($\mu \in \mathbb{Z}+\frac{1}{2}, n\in \mathbb{Z}$) satisfy the relations of 
\svir with central charge $C=\dfrac{3}{2} - 12\sigma^2 $:  
\begin{align}
& {\left[L_n, L_m\right]=(n-m) L_{n+m}+\frac{C}{12}\left(n^3-n\right) \delta_{n+m, 0}}, \label{eq: def rel sVir1}\\
& {\left[L_n, G_{\mu} \right]=\left(\frac{n}{2}-\mu\right) G_{n+\mu}}, \label{eq: def rel sVir2}\\
& {\left[G_{\mu}, G_{\nu} \right]_{+}=2 L_{\mu+\nu}+\frac{C}{3}\left(\mu^2-\frac{1}{4}\right) \delta_{\mu+\nu, 0}}. \label{eq: def rel sVir3}
\end{align}
\end{fact}

In terms of the free bosonic fields, the currents $T(z)$ and $G(z)$ can be written as follows.  

\begin{proposition}\label{prop: TG boson}
We have 
\begin{align}
&T(z) = :\frac{1}{2} \widetilde{\phi}^{\prime}(z)^2+\sigma \widetilde{\phi}^{\prime \prime}(z)
+\frac{1}{4} \left( \phi'(w)^2 -e^{2\phi(w)} -e^{-2\phi(w)} \right):,\label{eq: T boson}\\
&G(z)=:\frac{1}{\sqrt{-2}} \, \widetilde{\phi}'(z) \left( e^{ \phi(z)}-e^{ -\phi(z)} \right)
-\sqrt{-2}\, \sigma \phi'(z) \left( e^{\phi(z)} + e^{- \phi(z)}\right):. \label{eq: G boson}
\end{align}
\end{proposition}

\begin{proof}
(\ref{eq: G boson}) can be immediately shown by a direct calculation. 
(\ref{eq: T boson}) can be obtained by Lemma \ref{lem: fml BF corresp} in Appendix \ref{sec: fml bf corresp}. 
\end{proof}

The screening currents can be constructed as follows. 
They commute with \svir up to total derivatives.

\begin{definition}\label{def: svir screening}
Set $t_{\pm}=\sigma \pm \sqrt{\sigma^2+1}$. The screening currents $\mathsf{S}^{\pm}(z)$ are defined by 
\begin{align}
\mathsf{S}^{\pm}(z)= t_{\pm} \cdot \psi(z) :e^{t_{\pm} \widetilde{\phi}(z)}:. 
\end{align}
\end{definition}

\begin{fact}[\cite{KM1986null}]
It follows that 
\begin{align}
&[L_n, \mathsf{S}^{\pm}(z)]=\frac{\partial}{\partial z} \left( z^{n+1} \mathsf{S}^{\pm}(z)\right), \\
&[G_n, \mathsf{S}^{\pm}(z)]_+=\frac{\partial}{\partial z} \left( z^{n+\frac{1}{2}} \, e^{t_{\pm} \widetilde{\phi}(z) } \right).
\end{align}
\end{fact}

\subsection{Limit $q_i\rightarrow 1$}\label{subsec: lim}

We now consider the limit $q_i\rightarrow 1$. 
To take this limit, we set $\hbar =\log q_1$ and parametrize $q_1, q_2, q_3$ as 
\begin{align}
q_1= e^{\hbar}, \quad q_2=e^{(\beta-1)\hbar},\quad q_3=e^{-\beta \hbar}. 
\end{align} 
We study the limit $\hbar \rightarrow 0$ with $\beta$ fixed under this parametrization. 
Since the commutation relations among the generators $\boo_{i,n}$, $\cho$, $\zeo$, $\zcho$ depend on the parameters $q_1, q_2, q_3$, 
we identify them with the Heisenberg algebra $\mathcal{H}$ via the realization: 
\begin{align}
&\boo_{1, n} =  \frac{\left(1+q_1^n\right)\left(1+q_3^n\right)}{2(1+q_2^{-n})} \bon_n
+\frac{\sqrt{\beta}}{2}\left(1-q_1^n\right)  \bons_{n}, \qquad 
\boo_{1, -n} =\bon_{-n}+ \frac{1-q_3^n}{\sqrt{\beta}\, (1+q_2^{-n})} \bons_{-n}, \label{eq: realize boo1}\\
&\boo_{2, n} = -\frac{\left(1+q_1^n\right)\left(1+q_3^n\right)}{2(1+q_2^{-n})} \bon_{n}
+ \frac{\sqrt{\beta}}{2}\left(1-q_1^n\right)\bons_{n},\qquad 
\boo_{2, -n} = -\bon_{-n}+ \frac{1-q_3^n}{\sqrt{\beta}\, (1+q_2^{-n})} \bons_{-n}, \label{eq: realize boo2}\\
&\boo_{1,0}=-\boo_{2,0}=\bon_0, \quad \cho_1=-\cho_2=\chn, \quad 
\zeo=\frac{\sqrt{\beta}}{2} \bons_0, \quad 
\zcho=\sqrt{\beta}\, \chns \qquad (n>0). \label{eq: realize boo3}
\end{align}
By setting 
\begin{align}
&\varphi_q(z)=\sum_{n>0} \frac{1}{n} \bon_{-n} z^n-\frac{1}{2} \sum_{n>0} \frac{\left(1+q_1^n\right)\left(1+q_3^n\right)}{n\left(1+q_2^{-n}\right)} \bon_{n} z^{-n}+\bon_{0}(\log z)+\chn, \\
&\widetilde{\varphi}_q(z)=\sum_{n>0} \frac{1}{n} \cdot \frac{1-q_3^n}{\sqrt{\beta}(1+q_2^{-n})} \bons_{-n} z^n
-\frac{1}{2} \sum_{n>0} \frac{\sqrt{\beta}\, (1-q_1^n)}{n} \bons_{n} z^{-n}
+\frac{\sqrt{\beta}}{2}\,  \hbar \, \bons_0,
\end{align}
$\Lambda_i^{\pm}(z)$ can be written as 
\begin{align}
&\Lambda_1^{+}(z)=: e^{\varphi_q\left(q^{-1} z\right)} e^{\widetilde{\varphi}_q\left(q^{-1} z\right)}:, \quad 
\Lambda_1^-(z)=: e^{-\varphi_q(q z)} e^{-\widetilde{\varphi}_q(q z)}:, \\
&\Lambda_2^{+}(z)=: e^{-\varphi_q\left(q^{-1} z\right)} e^{\widetilde{\varphi}_q\left(q^{-1} z\right)}:, \quad 
\Lambda_2^-(z)=: e^{\varphi_q(q z)} e^{-\widetilde{\varphi}_q(q z)}:.
\end{align}

In this setting, \svir appears in the limit of $W_i(z)$.

\begin{theorem}\label{thm: expansion W}
The $\hbar$-expansions of $W_i(z)$ are of the forms 
\begin{align}
&W_1(z) = \sqrt{2} \, \widetilde{\psi}(z) +  \frac{\sqrt{-2\beta}}{2}  \,  G(z) \, z \, \hbar +O(\hbar^2), \label{eq: h-exp1}\\
&W_2(z) = \sqrt{2} \, \widetilde{\psi}(z) -  \frac{\sqrt{-2\beta}}{2}  \, G(z) \, z \, \hbar +O(\hbar^2). \label{eq: h-exp2}
\end{align}
Here, $G(z)$ is the fermionic current of \svir realized by (\ref{eq: rep of G(z)}) 
with $\di \sigma=\frac{1-\beta}{2\sqrt{\beta}}$. 
\end{theorem}

\begin{proof}
The limits of 
$\varphi_q(q^{\pm 1} z)$ and $\widetilde{\varphi}_q(q^{\pm 1} z)$ are given by
\begin{align}
\lim_{\hbar \rightarrow 0} \varphi_q(q^{\pm 1} z)=\phi(z),\quad 
\lim_{\hbar \rightarrow 0} \widetilde{\varphi}_q(q^{\pm 1} z)=0. 
\end{align}
Thus we have\footnote{
We denote by $\di \mathrm{Coeff}_{\hbar^n} A$ 
the coefficient of $\hbar^n$ in the $\hbar$-expansion of $A$.} 
\begin{align}
&\mathrm{Coeff}_{\hbar^0}  W_i(z)= \lim_{\hbar \rightarrow 0} W_i(z)
=:e^{\phi (z)}+e^{-\phi(z)}:= \sqrt{2} \, \widetilde{\psi}(z)\qquad 
(i=1,2). 
\end{align}
A direct calculation gives 
\begin{align}
&\lim_{\hbar \rightarrow 0}\frac{\partial}{\partial \hbar}  \varphi_q(q^{-1}z) 
=\frac{1-\beta}{2} \, \phi'(z) \cdot z,
\qquad 
\lim_{\hbar \rightarrow 0}\frac{\partial}{\partial \hbar}  \varphi_q(qz)
=-\frac{1-\beta}{2} \, \phi'(z) \cdot z,
\\
&\lim_{\hbar \rightarrow 0}\frac{\partial}{\partial \hbar}  \widetilde{\varphi}_q(q^{-1}z) 
=\lim_{\hbar \rightarrow 0}\frac{\partial}{\partial \hbar}  \widetilde{\varphi}_q(qz) 
=\frac{\sqrt{\beta}}{2} \, \widetilde{\phi}'(z) \cdot z. 
\end{align}
This leads to 
\begin{align}
&\lim_{\hbar \rightarrow 0}\frac{\partial}{\partial \hbar} \Lambda^{\pm}_1(z)
= :\frac{1-\beta}{2}\,  \phi'(z) e^{\pm \phi(z)}\cdot z
\pm \frac{\sqrt{\beta}}{2} \, \widetilde{\phi}'(z) e^{\pm \phi(z)}\cdot z:,\\
&\lim_{\hbar \rightarrow 0}\frac{\partial}{\partial \hbar} \Lambda^{\pm}_2(z)
= :-\frac{1-\beta}{2}\,  \phi'(z) e^{\mp \phi(z)}\cdot z
\pm \frac{\sqrt{\beta}}{2} \, \widetilde{\phi}'(z) e^{\mp \phi(z)}\cdot z:.
\end{align}
From these,  we obtain 
\begin{align}
&\mathrm{Coeff}_{\hbar^1} W_i(z)
=\lim_{\hbar \rightarrow 0}\frac{\partial}{\partial \hbar} W_i(z)\\
&\quad =\begin{cases}
\di :\frac{1-\beta}{2} \phi'(z) \left( e^{\phi(z)} + e^{- \phi(z)}\right)\cdot z
+ \frac{\sqrt{\beta}}{2} \, \widetilde{\phi}'(z) \left( e^{ \phi(z)}-e^{ -\phi(z)} \right) \cdot z:
&(i=1),\\
\di :-\frac{1-\beta}{2} \phi'(z) \left( e^{\phi(z)} + e^{- \phi(z)}\right)\cdot z
+ \frac{\sqrt{\beta}}{2} \, \widetilde{\phi}'(z) \left( -e^{ \phi(z)}+e^{ -\phi(z)} \right) \cdot z:
& (i=2).
\end{cases}\nonumber 
%&=\frac{\sqrt{-2\beta}}{2} \left( :\widetilde{\phi}'(z): \psi(z) + \frac{1-\beta}{\sqrt{\beta}} \psi'(z) \right)z,\\ %expression by psi
%&=-\frac{\sqrt{-2\beta}}{2} \left( :\widetilde{\phi}'(z): \psi(z) + \frac{1-\beta}{\sqrt{\beta}} \psi'(z) \right)z. %expression by psi
\end{align}
Comparing with Proposition \ref{prop: TG boson}, 
we see that the result coincides, and the proof is complete.
\end{proof}

\begin{remark}
By Theorem \ref{thm: expansion W}, the fermionic current $G(z)$ is given by the limit 
\begin{align}
G(z)=\lim_{\hbar \rightarrow 0}\frac{z^{-1}}{(q_1-1)\sqrt{-2\beta }}\left(W_{1}(z)-W_{2}(z)\right). 
\end{align}
This generates \svir with central charge $\di C=\frac{3}{2}-\frac{3(1-\beta)^2}{\beta}$.%
\footnote{There is a typo in the central charge of Theorem 4.2 in \cite{Ohkubo2025toward}. 
The correct value is $\di C=\frac{3}{2}-3\frac{(1-\kappa)^2}{\kappa}$. }
Operators corresponding to the generator $T(z)$ are discussed in Section \ref{sec: quad rel}. 
\end{remark}

\begin{remark}
The constant term $\widetilde{\psi}(z)$ appearing in the $\hbar$-expansions (\ref{eq: h-exp1}) and (\ref{eq: h-exp2})
anticommutes with $G(w)$:
\begin{align}
[\widetilde{\psi}(z), G(w)]_+=0. 
\end{align}
Hence, the entire algebra generated by $W_i(z)$ can be regarded as a $q$-deformation of 
$\mathsf{F}\oplus \svirm$, where 
$\mathsf{F}$ denotes the free fermion algebra generated by $\widetilde{\psi}_{\mu}$. 
\end{remark}

\section{Screening Currents}\label{sec: screening}

In this section, we discuss screening currents of $W_i(z)$. 
The algebra generated by $W_i(z)$ admits two screening currents, 
which are sums of two exponential terms such as 
the bosonic screenings of the quantum affine algebra $U_q(\widehat{\mathfrak{sl}}_2)$ \cite{Matsuo1994qdeformation}.
(See also \cite{AHKS2024quantum}.)
We could not construct screening currents consisting of a single exponential terms.

\begin{definition}\label{def: screening}
Set
\begin{align}
&\tau_+= \frac{1}{\beta},\quad \tau_{-}=-1, \\
&s_{+}=q_3,\quad s_{-}=q_1. 
\end{align}
Define the screening currents $S^+(z)$ and $S^-(z)$ by 
\begin{align}
S^{\pm}(z)=S^{\pm}_1(z)-S^{\pm}_2(z),
\end{align}
\begin{align}
S^{\pm}_1(z)=&\exp\left( \sum_{n>0}
\frac{s_{\pm}^{\frac{n}{2}}(1+q_2^n)}{q^n \left(1-s_{\pm}^{2n}\right)}(\boo_{1,-n}+s_{\pm}^n \boo_{2,-n})z^n \right)
\exp\left( \sum_{n>0}\frac{s_{\pm}^{\frac{n}{2}}(1+q_2^n)}{q^n \left(1-s_{\pm}^{2n}\right)}(s_{\pm}^n \boo_{1,n}+\boo_{2,n})z^{-n} \right)\nonumber \\
&\times e^{\cho_1+\tau_{\pm}\zcho}
z^{\boo_{1,0}+2\tau_{\pm}\zeo}, \\
S^\pm_2(z)=&\exp\left(\sum_{n>0} \frac{s_{\pm}^{\frac{n}{2}}(1+q_2^n)}{q^n \left(1-s_{\pm}^{2n}\right)}(s_{\pm}^n \boo_{1,-n}+\boo_{2,-n})z^n\right)
\exp\left( \sum_{n>0} \frac{s_{\pm}^{\frac{n}{2}}(1+q_2^n)}{q^n \left(1-s_{\pm}^{2n}\right)} (\boo_{1,n}+s_{\pm}^n \boo_{2,n}) z^{-n}\right) \nonumber\\
&\times e^{-\cho_1+\tau_{\pm}\zcho} 
z^{-\boo_{1,0}+2\tau_{\pm}\zeo}. 
\end{align}
\end{definition}

Apart from slight modifications to the zero modes, 
these screening currents $S^{\pm}(z)$ are the same 
as those constructed in \cite{Zenkevich2019gln}. 
As for the operator product formulas among $S^{\pm}_i(z)$, see Appendix \ref{sec: OPE screening}. 
These screening currents $S^{\pm}(w)$ anticommute with $W_i(z)$ up to the total difference of an operator.

\begin{proposition}\label{prop: S W rel}
We obtain
\begin{align}
&[W_1(z), S^{\pm}(w)]_+
=w^{-1}(T_{s_{\pm},w}-1)\delta\left(\frac{w}{s_{\pm}^{1/2}z}\right):\Lambda^+_{1}(s_{\pm}^{-1/2}w)S^{\pm}_2(w):, \label{eq: W_1 S}\\
&[W_2(z), S^{\pm}(w)]_+
=w^{-1}(1-T_{s_{\pm},w})\delta\left(\frac{w}{s_{\pm}^{1/2}z}\right):\Lambda^+_{2}(s_{\pm}^{-1/2}w)S^{\pm}_1(w):.\label{eq: W_2 S}
\end{align} 
%\begin{align}%i+ and i-
%&W_i(z)S^+(w)+S^+(w)W_i(z)
%=(-1)^{i+1}w^{-1}(1-T_{q_3,w})\delta\left(\frac{w}{q_3^{1/2}z}\right):\Lambda^+_{i}(q_3^{-1/2}w)S^+_i(w):,\\
%&W_i(z)S^{-}(w)+S^{-}(w)W_i(z)
%=(-1)^{i+1}w^{-1}(1-T_{q_1,w})\delta\left(\frac{w}{q_1^{1/2}z}\right):\Lambda^+_{i}(q_1^{-1/2}w)S^{-}_i(w):.
%\end{align}
%\begin{align}%All
%W_i(z)S^{\pm}(w)+S^{\pm}(w)W_i(z)
%=(-1)^{i+1}w^{-1}(1-T_{s_{\pm},w})\delta\left(\frac{w}{s_{\pm}^{1/2}z}\right):\Lambda^+_{i}(s_{\pm}^{-1/2}w)S^{\pm}_i(w):.
%\end{align} 
Here, $T_{p,w}$ is the difference operator defined by $T_{p,w}f(w)=f(pw)$. 
\end{proposition}

\begin{proof}
First, we show (\ref{eq: W_1 S}). 
The operator products among $\Lambda^{\pm}_1(z)$ and $S^{\pm}_i (w)$ take the form 
\begin{align}
&\Lambda^+_{1}(z)S^{\pm}_{1}(w)=\left(1-\frac{q_2 s_{\pm}^{1/2}w}{z} \right) \frac{z}{q_2 s_{\pm}^{1/2}}:\Lambda^+_{1}(z)S^{\pm}_{1}(w):,\\
&S^{\pm}_{1}(w)\Lambda^+_{1}(z)=\left(1-\frac{z}{q_2 s_{\pm}^{1/2}w} \right) w :S^{\pm}_{1}(w)\Lambda^+_{1}(z):,\\
&\Lambda^+_{1}(z)S^{\pm}_{2}(w)=\frac{1}{1-w/(s_{\pm}^{1/2}z)}s_{\pm}^{-1/2}z^{-1}
:\Lambda^+_{1}(z)S^{\pm}_{2}(w):,\\
&S^{\pm}_{2}(w)\Lambda^+_{1}(z)=\frac{1}{1-s_{\pm}^{1/2}z/w}w^{-1}:S^{\pm}_{2}(w)\Lambda^+_{1}(z):,\\
&\Lambda^-_{1}(z)S^{\pm}_{1}(w)=\frac{1}{1-s_{\pm}^{1/2}w/z}s_{\pm}^{1/2}z^{-1}:\Lambda^-_{1}(z)S^{\pm}_{1}(w):,\\
&S^{\pm}_{1}(w)\Lambda^-_{1}(z)=\frac{1}{1-z/(s_{\pm}^{1/2}w)}w^{-1}:S^{\pm}_{1}(w)\Lambda^-_{1}(z):,\\
&\Lambda^-_{1}(z)S^{\pm}_{2}(w)=\left(1-\frac{w}{q_2 s_{\pm}^{1/2}z} \right) q_2 s_{\pm}^{1/2} z
:\Lambda^-_{1}(z)S^{\pm}_{2}(w):,\\
&S^{\pm}_{2}(w)\Lambda^-_{1}(z)=\left(1-\frac{q_2 s_{\pm}^{1/2}z}{w} \right) w
:S^{\pm}_{2}(w)\Lambda^-_{1}(z):.
\end{align}
Thus, we have 
\begin{align}
&\Lambda^+_{1}(z)S^{\pm}_{1}(w)+S^{\pm}_{1}(w)\Lambda^+_{1}(z)=0,\\
&\Lambda^+_{1}(z)S^{\pm}_{2}(w)+S^{\pm}_{2}(w)\Lambda^+_{1}(z)
=w^{-1}\delta\left(\frac{w}{s_{\pm}^{1/2}z}\right):\Lambda^+_{1}(s_{\pm}^{-1/2}w)S^{\pm}_{2}(w):,\\
&\Lambda^-_{1}(z)S^{\pm}_{1}(w)+S^{\pm}_{1}(w)\Lambda^-_{1}(z)
=w^{-1}\delta\left(\frac{s_{\pm}^{1/2}w}{z}\right):\Lambda^-_{1}(s_{\pm}^{1/2}w)S^{\pm}_{1}(w):, \\
&\Lambda^-_{1}(z)S^{\pm}_{2}(w)+S^{\pm}_{2}(w)\Lambda^-_{1}(z)=0. 
\end{align}
By the relation $T_{s_{\pm},w}:\Lambda^+_{1}(s_{\pm}^{-1/2}w)S^{\pm}_{2}(w):=:\Lambda^-_{1}(s_{\pm}^{1/2}w)S^{\pm}_{1}(w):$, 
we obtain (\ref{eq: W_1 S}).

Next, we show (\ref{eq: W_2 S}). 
The operator products among $\Lambda^{\pm}_2(z)$ and $S^{\pm}_i (w)$ take the form
\begin{align}
&\Lambda^+_{2}(z)S^{\pm}_{1}(w)=\frac{1}{1-w/( s_{\pm}^{1/2}z)} s_{\pm}^{-1/2}z^{-1}:\Lambda^+_{2}(z)S^{\pm}_{1}(w):,\\
&S^{\pm}_{1}(w)\Lambda^+_{2}(z)=\frac{1}{1- s_{\pm}^{1/2}z/w}w^{-1}:S^{\pm}_{1}(w)\Lambda^+_{2}(z):,\\
&\Lambda^+_{2}(z)S^{\pm}_{2}(w)=\left(1-\frac{q_2 s_{\pm}^{1/2} w}{z} \right) \frac{z}{q_2 s_{\pm}^{1/2}}
:\Lambda^+_{2}(z)S^{\pm}_{2}(w):,\\
&S^{\pm}_{2}(w)\Lambda^+_{2}(z)=\left(1-\frac{z}{q_2 s_{\pm}^{1/2} w} \right) w 
:S^{\pm}_{2}(w)\Lambda^+_{2}(z):,\\
&\Lambda^-_{2}(z)S^{\pm}_{1}(w)=\left(1-\frac{w}{q_2 s_{\pm}^{1/2} z} \right) q_2 s_{\pm}^{1/2} z
:\Lambda^-_{2}(z)S^{\pm}_{1}(w):,\\
&S^{\pm}_{1}(w)\Lambda^-_{2}(z)=\left(1-\frac{q_2 s_{\pm}^{1/2} z}{ w} \right) w:S^{\pm}_{1}(w)\Lambda^-_{2}(z):,\\
&\Lambda^-_{2}(z)S^{\pm}_{2}(w)=\frac{1}{1- s_{\pm}^{1/2}w/z} s_{\pm}^{1/2}z^{-1}:\Lambda^-_{2}(z)S^{\pm}_{2}(w):,\\
&S^{\pm}_{2}(w)\Lambda^-_{2}(z)=\frac{1}{1-z/( s_{\pm}^{1/2}w)}w^{-1}:S^{\pm}_{2}(w)\Lambda^-_{2}(z):. 
\end{align}
Thus, we have 
\begin{align}
&\Lambda^+_{2}(z)S^{\pm}_{1}(w)+S^{\pm}_{1}(w)\Lambda^+_{2}(z)=w^{-1}\delta\left( \frac{w}{s_{\pm}^{1/2}z}\right)
:\Lambda^+_{2}(s_{\pm}^{-1/2}w)S^{\pm}_{1}(w):, \\
&\Lambda^+_{2}(z)S^{\pm}_{2}(w)+S^{\pm}_{2}(w)\Lambda^+_{2}(z)
=0,\\
&\Lambda^-_{2}(z)S^{\pm}_{1}(w)+S^{\pm}_{1}(w)\Lambda^-_{2}(z)=0, \\
&\Lambda^-_{2}(z)S^{\pm}_{2}(w)+S^{\pm}_{2}(w)\Lambda^-_{2}(z)=w^{-1} \delta\left( \frac{s_{\pm}^{1/2}w}{z} \right)
:\Lambda^-_{2}(s_{\pm}^{1/2}w)S^{\pm}_{2}(w):. 
\end{align}
By the relation $T_{s_{\pm},w}:\Lambda^+_{2}(s_{\pm}^{-1/2}w)S^{\pm}_{1}(w):  =  :\Lambda^-_{2}(s_{\pm}^{1/2}w)S^{\pm}_{2}(w):$, 
we obtain (\ref{eq: W_2 S}).
\end{proof}

The degenerate limits of the screening currents $S^+(z)$ and $S^{-}(z)$ coincide with the ones of \svir. 

\begin{comment}
In what follows, we assume $\di \sigma=\frac{1-\beta}{2\sqrt{\beta}}$. 
In this case, the parameter $t_{\pm}$ in Definition \ref{def: svir screening} can be written as 
\begin{align}
t_{+}=\frac{1}{\sqrt{\beta}}, \quad t_{-}=-\sqrt{\beta}. 
\end{align}
Since the square root is a double-valued function, 
$t_{+}$ and $t_{-}$ may be interchanged. 
We shall adopt the above assignment as our convention.
\end{comment}

\begin{theorem}
Under the realization (\ref{eq: realize boo1})--(\ref{eq: realize boo3}), 
we obtain 
\begin{align}
\lim_{\hbar \rightarrow 0} S^{\pm}(z)=\frac{\sqrt{-2}}{t_{\pm}}\, \mathsf{S}^{\pm}(z). 
\end{align}
Here, $\mathsf{S}^{\pm}(z)$ are the screening currents of \svir (Definition \ref{def: svir screening}), 
with $\di \sigma=\frac{1-\beta}{2\sqrt{\beta}}$. 
The parameters $t_{\pm}$ are assigned as\footnote{
Although $t_{\pm}$ are formally given by $t_{\pm}=\sigma \pm \sqrt{\sigma^2+1}$, 
since the square root is multivalued, we fix the value of $t_{\pm}$ as above to match the limit.}%
\begin{align}
t_{+}=\frac{1}{\sqrt{\beta}}, \quad t_{-}=-\sqrt{\beta}. 
\end{align}
\end{theorem}

\proof
In terms of the Heisenberg algebra $\mathcal{H}$, 
the screening currents can be written as 
\begin{align}
&S^+_1(z)=A^+(z)B(z), \quad S^+_2(z)=A^-(z)B(z), \\
&S^-_1(z)=C^+(z)D(z),\quad S^-_2(z)=C^-(z)D(z),
\end{align}
where 
\begin{align}
&A^{\pm}(z)=\exp\left( \pm \sum_{n>0}
\frac{\left(1+q_2^{-n}\right)}{nq_1^{\frac{n}{2}} \left(1+q_3^n\right)}\bon_{-n} z^n \right)
\exp\left( \mp \sum_{n>0}\frac{\left(1+q_1^n\right)}{2n  q_1^{\frac{n}{2}}}\bon_nz^{-n} \right)
e^{\pm\chn}z^{\pm\bon_0}, \\
&B(z)=\exp\left(\sum_{n>0}
\frac{1}{n \sqrt{\beta } q_1^{\frac{n}{2}}}\bons_{-n}z^n \right)
\exp\left(\sum_{n>0}
\frac{\sqrt{\beta } \left(1-q_1^n\right) \left(1+q_2^{-n}\right)}{2n q_1^{\frac{n}{2}} \left(1-q_3^n\right)}\bons_n z^{-n} \right)
e^{\frac{1}{\sqrt{\beta}}\chns}
z^{\frac{1}{\sqrt{\beta}}\bons_0}, \\
&C^{\pm}(z)=\exp\left(\pm \sum_{n>0}
\frac{\left(1+q_2^{-n}\right)}{n \left(1+q_1^n\right) q_3^{\frac{n}{2}}}\bon_{-n}z^n \right)
\exp\left( \mp \sum_{n>0}\frac{ \left(1+q_3^n\right)}{2 n q_3^{\frac{n}{2}}}\bon_n z^{-n} \right)
e^{\pm\chn}z^{\pm\bon_0},\\
&D(z)=\exp\left( \sum_{n>0}
\frac{\left(1-q_3^n\right) }{n \sqrt{\beta } \left(1-q_1^n\right) q_3^{\frac{n}{2}}}\bons_{-n}z^n \right) 
\exp\left(\sum_{n>0}\frac{\sqrt{\beta } \left(1+q_2^{-n}\right)}{2 n q_3^{\frac{n}{2}}}\bons_n z^{-n} \right)
e^{-\sqrt{\beta} \, \chns} z^{-\sqrt{\beta}\, \bons_0}.
\end{align}
This yields 
\begin{align}
&\lim_{\hbar \rightarrow 0} S^{\pm}(z)
=:(e^{\phi(z)}-e^{-\phi(z)})e^{t_{\pm} \widetilde{\phi}(z)}:
=\sqrt{-2}\, \psi(z) :e^{t_{\pm}\widetilde{\phi}(z)}:
=\frac{\sqrt{-2}}{t_{\pm}}\, \mathsf{S}^{\pm}(z). 
%&\lim_{\hbar \rightarrow 0} S^+(z)=:(e^{\phi(z)}-e^{-\phi(z)})e^{\frac{1}{\sqrt{\beta}}\widetilde{\phi}(z)}:
%=\sqrt{-2}\, \psi(z) :e^{\frac{1}{\sqrt{\beta}}\widetilde{\phi}(z)}:
%=\frac{\sqrt{-2}}{t_{+}}\, \mathsf{S}^{+}(z) ,\\
%&\lim_{\hbar \rightarrow 0} S^-(z)=:(e^{\phi(z)}-e^{-\phi(z)})e^{-\sqrt{\beta}\widetilde{\phi}(z)}:
%=\sqrt{-2}\, \psi(z) :e^{-\sqrt{\beta}\widetilde{\phi}(z)}:
%=\frac{\sqrt{-2}}{t_{-}}\, \mathsf{S}^{-}(z). 
\end{align}
\qed 

\begin{remark}
In order to study the correspondence with the undeformed screening currents $\mathsf{S}^{\pm}(z)$, 
we construct $S^{\pm}(z)$ 
using zero modes such as $\zeo$ and $\zcho$. 
However, for a rigorous treatment including integration contours, 
it should be more appropriate to replace the zero modes by suitable ratios of theta functions as in \cite{JLMP1996Lukyanov}.
\end{remark}

\section{Quadratic relations}\label{sec: quad rel}

In this section, we establish the quadratic relations of the generators $W_i(z)$. 
Depending on the choice of structure functions, several relations can be derived for $W_i(z)$.
We begin with the simplest quadratic relations, which take the following form.

\begin{proposition}\label{prop: rel WW}
We have 
\begin{align}
&W_i(z)W_i(w)+W_i(w)W_i(z)
=qz^{-1}\delta\left( \frac{q_2w}{z} \right) +q^{-1}z^{-1}\delta\Big( \frac{w}{q_2z} \Big)\qquad (i=1,2),\label{eq: WW relation i=j}\\
&f\left(\frac{w}{z}\right)zW_i(z)W_j(w)-f\left(\frac{z}{w}\right)wW_j(w)W_i(z)=0\qquad (i\neq j). \label{eq: WW relation ineqj}
\end{align}
Here, we have set 
\begin{equation}
f(z)=\exp\left( -\sum_{n>0}\frac{q_1^n+q_3^n}{n(1+q_2^{-n})} z^n\right).
\end{equation}
\end{proposition}

By using the Fourier components $W_{i,\mu}$ in the mode expansion
$\di W_{i}(z)=\sum_{\mu \in \mathbb{Z}+\frac{1}{2}}W _{i,\mu} z^{-\mu-\frac{1}{2}}$ 
and the constants $f_{\ell}$ defined by $\di f(z)=\sum_{\ell =0} ^{\infty} f_{\ell} \, z^{\ell}$,
the relations of Proposition \ref{prop: rel WW}  can be written as 
\begin{gather}
[W_{i,r}, W_{i,s}]_+= (q_2^r+q_2^{-r})\delta_{r+s,0}, \\
W_{i,\mu}W_{j,\nu} -f_1 \, W_{j,\nu}W_{i,\mu}
= - \sum_{\ell =1}^{\infty} (f_{\ell } \, W_{i,\mu-\ell}  W_{j,\nu+\ell} -f_{\ell +1} W_{j,\nu-\ell} W_{i,\mu+\ell}) +W_{j,\nu+1}W_{i,\mu-1}. \label{eq: WW rel ineqj mode}
\end{gather}

\begin{proof}[Proof of Proposition \ref{prop: rel WW}]
The operator products among $\Lambda^{\pm}_i(z)$ take the form
\begin{align}
&\Lambda^{\pm}_i(z)\Lambda^{\pm}_i(w)=(1-w/z)q^{\mp 1} z :\Lambda^{\pm}_i(z)\Lambda^{\pm}_i(w):, \label{eq: OPE Lampm i=j 1}\\
&\Lambda^{\pm}_i(z)\Lambda^{\mp}_i(w)=\frac{q^{\pm 1}z^{-1}}{(1-q_2^{\pm 1} w/z)} :\Lambda^{\pm}_i(z)\Lambda^{\mp}_i(w): \quad (i=1,2), \label{eq: OPE Lampm i=j 2}
\end{align}
\begin{align}
&\Lambda^{\pm}_i(z)\Lambda^{\pm}_j(w)=f^{-1}(w/z) q^{\pm 1} z^{-1} :\Lambda^{\pm}_i(z)\Lambda^{\pm}_j(w):,\label{eq: OPE Lampm ineqj 1}\\
&\Lambda^{\pm}_i(z)\Lambda^{\mp}_j(w)=f(q_2^{\pm 1} w/z)q^{\mp 1} z :\Lambda^{\pm}_i(z)\Lambda^{\mp}_j(w):
\quad (i\neq j). \label{eq: OPE Lampm ineqj 2}
\end{align}
\begin{comment} %separated form
\begin{align}
&\Lambda^+_i(z)\Lambda^+_i(w)=(1-w/z)q^{-1}z :\Lambda^+_i(z)\Lambda^+_i(w):,\\
&\Lambda^-_i(z)\Lambda^-_i(w)=(1-w/z)q z :\Lambda^-_i(z)\Lambda^-_i(w):,\\
&\Lambda^+_i(z)\Lambda^-_i(w)=\frac{q}{(1-q_2w/z)z} :\Lambda^+_i(z)\Lambda^-_i(w):,\\
&\Lambda^-_i(z)\Lambda^+_i(w)=\frac{q^{-1}}{(1-q_2^{-1}w/z)z} :\Lambda^-_i(z)\Lambda^+_i(w):
\end{align}
If $i\neq j$, 
\begin{align}
&\Lambda^+_i(z)\Lambda^+_j(w)=f^{-1}(w/z):\Lambda^+_i(z)\Lambda^+_j(w):\times q z^{-1},\\
&\Lambda^-_j(z)\Lambda^-_i(w)=f^{-1}(w/z) :\Lambda^-_j(z)\Lambda^-_i(w):\times q^{-1} z^{-1},\\
&\Lambda^+_i(z)\Lambda^-_j(w)=f(q_2w/z) :\Lambda^+_i(z)\Lambda^-_j(w):\times q^{-1} z, \\
&\Lambda^-_j(z)\Lambda^+_i(w)=f(q_2^{-1}w/z):\Lambda^-_j(z)\Lambda^+_i(w): \times q z
\end{align}
\end{comment}
(\ref{eq: OPE Lampm i=j 1}) and (\ref{eq: OPE Lampm i=j 2}) yield
\begin{align}
\Lambda_i^{ \pm}(z) \Lambda_i^{ \pm}(w)+\Lambda_i^{ \pm}(w) \Lambda_i^{ \pm}(z)&=0,\\
\Lambda_i^{\pm}(z) \Lambda^{\mp}_i(w)+\Lambda_i^{\mp}(w) \Lambda_i^{\pm}(z)
&=q^{\pm 1} z^{-1} \delta\left(q_2^{\pm 1} w / z\right): \Lambda_i^{\pm}\left(q_2^{\pm} w\right) \Lambda_i^{\mp}(w):.
%&=q^{\pm 1} z^{-1} \delta\left(q_2^{\pm 1} w / z\right). 
\end{align}
Therefore by using $:\Lambda_i^{\pm}\left(q_2^{\pm} w\right) \Lambda_i^{\mp}(w):=1$, 
we obtain (\ref{eq: WW relation i=j}). 

Since we have 
\begin{align}
f(z)f(q_2^{\pm} z)=(1-q_1^{\mp 1}z)(1-q_3^{\mp 1}z),
\end{align}
(\ref{eq: OPE Lampm ineqj 1}) and (\ref{eq: OPE Lampm ineqj 2}) yield
\begin{align}
&f(w/z)z \Lambda^{\pm}_i(z)\Lambda^{\pm}_j(w)-f(z/w)w \Lambda^{\pm}_j(w)\Lambda^{\pm}_i(z)= 0,\\
&f(w/z)z \Lambda^{\pm}_i(z)\Lambda^{\mp}_j(w)-f(z/w)w \Lambda^{\mp}_j(z)\Lambda^{\pm}_i(w)=0. 
\end{align}
These lead to (\ref{eq: WW relation ineqj}). 
\end{proof}

Note that the relation (\ref{eq: WW rel ineqj mode}) is not sufficient to perform the normal ordering of the Fourier components $W_{i,\mu}$, because of the last term $W_{j,\nu+1}W_{i,\mu-1}$. 
Hence, the highest weight representations cannot be constructed solely from the above relations. 
In order to perform the normal ordering, 
we need to introduce additional generators and formulate modified relations.

\begin{definition}\label{def: T op.}
For a non-zero complex parameter $\xi$, we define 
\begin{align}
\mathcal{T}_{ij}(\xi;z)=&M^{(1)}_{ij}(\xi;z)+M^{(2)}_{ij}(\xi;z)+z^2 M^{(3)}_{ij}(\xi;z)+z^2 M^{(4)}_{ij}(\xi;z)
\quad (i\neq j). 
\label{eq: def T(xi,w)}
\end{align}
Here we set
\begin{align}
&M^{(1)}_{ij}(\xi;z)=q : \Lambda_i^{+}(\xi z) \Lambda_j^{+}(z):, \qquad 
M^{(2)}_{ij}(\xi;z)=q^{-1} : \Lambda_i^{-}(\xi z) \Lambda_j^{-}(z):,\\
&M^{(3)}_{ij}(\xi;z)=q(1-q_1 \xi)\left(1-q_3 \xi\right) : \Lambda_i^{+}(\xi z) \Lambda_{j}^{-}(z):,\\
&M^{(4)}_{ij}(\xi;z)=q^{-1}\left(1-q_1^{-1} \xi\right)\left(1-q_3^{-1} \xi\right)  \Lambda_i^{-}(\xi z) \Lambda_j^{+}(z):.
\end{align}
\begin{comment} %full form 
\begin{align}
T_{ij}(\xi;w)=&
q : \Lambda_i^{+}(\xi w) \Lambda_j^{+}(w):+q^{-1} : \Lambda_i^{-}(\xi w) \Lambda_j^{-}(w): \nonumber \\
&+q(1-q_1 \xi)\left(1-q_3 \xi\right) w^2: \Lambda_i^{+}(\xi w) \Lambda_{j}^{-}(w): \nonumber \\
& +q^{-1}\left(1-q_1^{-1} \xi\right)\left(1-q_3^{-1} \xi\right) w^2 : \Lambda_i^{-}(\xi w) \Lambda_j^{+}(w):.
%\label{eq: def T(xi,w)}
\end{align}
\end{comment}
\end{definition}

This generator satisfies the symmetry 
$\mathcal{T}_{ij}(\xi; \xi^{-\frac{1}{2}}w)=\mathcal{T}_{ji}(\xi^{-1};\xi^{\frac{1}{2}} w)$. 
Accordingly, we often fix the indices to $\mathcal{T}_{12}(\xi;z)$ and use the shorthand notation
\begin{align}
\mathcal{T}(\xi;z)=\mathcal{T}_{12}(\xi;z), \quad 
M_k(\xi;z)=M^{(k)}_{12}(\xi;z). \label{eq: T shorthand}
\end{align}
We also note that if $\xi=q_1^{\pm 1}$ or $q_3^{\pm 1}$, 
either $M^{(3)}_{ij}$ or $M^{(4)}_{ij}$ vanishes. 
Furthermore, we define the structure function $f(\xi; z)$ by 
\begin{align}
f(\xi; z)
=\frac{1}{1-\xi z}f\left(z\right)
=\exp\left\{ \sum_{n=1}^{\infty} \left( \xi^n -\frac{q_1^n+q_3^n}{(1+q_2^{-n})} \right) \frac{z^n}{n} \right\}.
\end{align}
For example, depending on the value of $\xi$,  
the structure function $f(\xi; z)$  takes the form
\begin{align}
&f(1;z)=\exp\left( \sum_{n>0}\frac{(1-q_1^{n})(1-q_3^{n})}{n(1+q_2^{-n})} z^n\right),\\
&f(q_1;z)=\exp\left(- \sum_{n>0}\frac{(1-q_1^{2n})q_3^n}{n(1+q_2^{-n})} z^n\right), \quad 
f(q_1^{-1};z)=\exp\left(\sum_{n>0}\frac{(1-q_1^{2n})q_1^{-n}}{n(1+q_2^{-n})} z^n\right),\\
&f(q_3;z)=\exp\left(- \sum_{n>0}\frac{(1-q_3^{2n})q_1^n}{n(1+q_2^{-n})} z^n\right), \quad 
f(q_3^{-1};z)=\exp\left(\sum_{n>0}\frac{(1-q_3^{2n})q_3^{-n}}{n(1+q_2^{-n})} z^n\right).
\end{align}

\begin{theorem}\label{thm: rel WWT}
Let $i \neq j$. Then it follows that 
\begin{align}
%&\frac{\xi}{1-\xi w/z}f\left(\frac{w}{z}\right) W_i(z) W_j(w) %rational form
%+\frac{1}{1-\xi^{-1} z/w}f\left(\frac{z}{w} \right) W_j(w) W_i(z)
&\xi \cdot f\left(\xi; \frac{w}{z}\right) W_i(z) W_j(w) 
+f\left(\xi^{-1} ; \frac{z}{w} \right) W_j(w) W_i(z)
= \delta \left( \frac{\xi w}{z}\right) w^{-1} \mathcal{T}_{ij}(\xi; w). \label{eq: WWT}
\end{align}
\end{theorem}

Define the operators $\mathcal{T}_{ij; n}^{(\xi)}$ and 
the constants $f_{\ell}^{(\xi)}$ ($n\in \mathbb{Z}$, $\ell \in \mathbb{Z}_{\geq 0}$) by 
\begin{align}
&\mathcal{T}_{ij}(\xi;z)=\sum_{n\in \mathbb{Z}} \mathcal{T}_{ij;n}^{(\xi)} z^{-n}, \quad 
f\left(\xi; z \right) =\sum_{\ell=0}^{\infty} f_{\ell}^{(\xi)} z^{\ell }.
\end{align}
The relation (\ref{eq: WWT}) is equivalent to 
\begin{align}
\xi \, W_{i,\mu} W_{j,\nu} +W_{j,\nu} W_{i,\mu}
=&-\sum_{\ell =1}^{\infty} \left( \xi\, f_{\ell}^{(\xi)} W_{i,\mu-\ell} W_{j, \nu +\ell}
+f_{\ell}^{(\xi^{-1})} W_{j,\nu-\ell} W_{i, \mu +\ell}\right) \nonumber  \\
&+\xi^{\mu +\frac{1}{2}} \mathcal{T}_{ij;\mu+\nu}^{(\xi)}. 
\end{align}
By this quadratic relation, we can perform the normal ordering of $W_{i,\mu}$. 

\begin{proof}[Proof of Theorem \ref{thm: rel WWT}]
By the operator products (\ref{eq: OPE Lampm i=j 1}) and (\ref{eq: OPE Lampm i=j 2}), it follows that 
\begin{align}
&\xi \cdot f\left(\xi ;\frac{w}{z}\right) \Lambda_i^{\pm}(z) \Lambda_j^{\pm}(w)
=\frac{\xi}{1-\xi w/z} \cdot q^{\pm 1} z^{-1}:\Lambda_i^{\pm}(z) \Lambda_j^{\pm}(w):,\\
&f\left(\xi^{-1} ;\frac{z}{w}\right) \Lambda_j^{\pm}(w) \Lambda_i^{\pm}(z)
=\frac{1}{1-\xi^{-1} z/w} \cdot q^{\pm 1} w^{-1}:\Lambda_j^{\pm}(w) \Lambda_i^{\pm}(z):.
\end{align}
Thus, we have
\begin{gather}
\xi\cdot f\left(\xi ;\frac{w}{z}\right) \Lambda_i^{\pm}(z) \Lambda_j^{\pm}(w)
+f\left(\xi^{-1} ;\frac{z}{w}\right) \Lambda_j^{\pm}(w) \Lambda_i^{\pm}(z)\nonumber \\
=q^{\pm 1} w^{-1} \delta\left(\frac{\xi w}{z}\right): \Lambda_i^{\pm}\left(\xi w\right) \Lambda_j^{\pm}(w):.
\label{eq: lam pmpm}
\end{gather}
By the operator products (\ref{eq: OPE Lampm ineqj 1}) and (\ref{eq: OPE Lampm ineqj 2}), it follows that 
\begin{align}
&\xi\cdot f\left(\xi ;\frac{w}{z}\right) \Lambda_i^{\pm}(z) \Lambda_j^{\mp}(w)
=\frac{\xi\left(1-q_1^{\mp 1} w/z\right)\left(1-q_3^{\mp 1} w/z\right)}{1-\xi w/z} q^{\mp 1} z: \Lambda_i^{\pm}(z) \Lambda_j^{\mp }(w):, \\
&f\left(\xi^{-1} ;\frac{z}{w}\right) \Lambda_j^{\mp }(w) \Lambda_i^{\pm}(z)
=\frac{\left(1-q_1^{\pm 1} z/w\right)\left(1-q_3^{\pm 1} z/w\right)}{1-\xi^{-1} z/w} q^{\pm 1} w:\Lambda_j^{\mp }(w) \Lambda_i^{\pm}(z):.
\end{align}
Thus, we have
\begin{align}
&\xi \cdot f\left(\xi ;\frac{w}{z}\right) \Lambda_i^{\pm}(z) \Lambda_j^{\mp }(w)
+f\left(\xi^{-1} ;\frac{z}{w}\right) \Lambda_j^{\mp }(w) \Lambda_i^{\pm}(z)\nonumber \\
&=q^{\pm 1}\left(1-q_1^{\pm 1} \xi\right)\left(1-q_3^{\pm 1} \right) \delta\left(\frac{\xi w}{z}\right) w
: \Lambda_i^{\pm}(\xi w) \Lambda_j^{\mp}(w):.
\label{eq: lam pmmp}
\end{align}
The relations (\ref{eq: lam pmpm}) and (\ref{eq: lam pmmp}) yield (\ref{eq: WWT}). 
\end{proof}

By choosing the parameter $\xi$ appropriately, 
we obtain two pairs of commuting operators. 
As will be mentioned in Remark \ref{rem: commuting vir} below, 
these pairs degenerate into two nontrivial commuting Virasoro algebras in $\mathsf{F} \oplus \svirm$. 

\begin{theorem}
We have 
\begin{align}
[\mathcal{T}(q_1;z), \mathcal{T}(q_3;w)]
=[\mathcal{T}(q_1^{-1};z), \mathcal{T}(q_3^{-1};w)]=0.
\end{align}
\end{theorem}

\begin{proof}
We prove only $[\mathcal{T}(q_1;z), \mathcal{T}(q_3;z)]=0$. 
The commutation relation $[\mathcal{T}(q_1^{-1};z), \mathcal{T}(q_3^{-1};w)]=0$ can be shown similarly. 
The operator products among $M_k(q_1;z)$ and $M_{\ell}(q_3;w)$ take the form 
\begin{align}
&M_k(q_1;z)M_{\ell}(q_3;w)=:M_k(q_1;z)M_{\ell}(q_3;w):,\\
&M_{\ell}(q_3;w)M_k(q_1;z)=:M_{\ell}(q_3;w)M_k(q_1;z):\quad (\forall k, \ell \in \{1,2\}), 
\end{align}
\begin{align}
&M_1(q_1;z)M_3(q_3;w)=\frac{1-q_1^{-2}w/z}{1-w/z}\cdot q_1^2 :M_1(q_1;z)M_3(q_3;w):, \\
&M_3(q_3;w)M_1(q_1;z)=\frac{1-q_1^{2}z/w}{1-z/w} :M_3(q_3;w)M_1(q_1;z):, \\
&M_2(q_1;z)M_3(q_3;w)=\frac{1-q_2^{-1} w/z}{1-q_1^{-1}q_3w/z}\cdot q_1^{-2} :M_2(q_1;z)M_3(q_3;w):, \\
&M_3(q_3;w)M_2(q_1;z)=\frac{1-q_2 z/w}{1-q_1q_3^{-1} z/w} :M_3(q_3;w)M_1(q_1;z):, 
\end{align}
\begin{align}
&M_3(q_1;z)M_1(q_3;w)=\frac{1-q_3^{2}w/z}{1-w/z} :M_3(q_1;z)M_1(q_3;w):, \\
&M_1(q_3;w)M_3(q_1;z)=\frac{1-q_3^{-2}z/w}{1-z/w} \cdot q_3^2:M_1(q_3;w)M_3(q_1;z):, \\
&M_3(q_1;z)M_2(q_3;w)=\frac{1-q_2w/z}{1-q_1^{-1}q_3w/z} :M_3(q_1;z)M_2(q_3;w):, \\
&M_2(q_3;w)M_3(q_1;z)=\frac{1-q_2^{-1}z/w}{1-q_1q_3^{-1}z/w} \cdot q_3^{-2} :M_2(q_3;w)M_3(q_1;z):, 
\end{align}
\begin{align}
&M_3(q_1;z)M_3(q_3;w)\\
&= \left(1-\frac{w}{q_1^2 z}\right) \left(1-\frac{q_3^2 w}{z}\right) \left(1-\frac{q_2w}{z}\right)
\left(1-\frac{w}{q_2z}\right) \cdot q_1^2 z^4 :M_3(q_1;z)M_3(q_3;w):, \nonumber \\
&M_3(q_3;w)M_3(q_1;z)\\
&=\left(1-\frac{q_1^2 z}{w}\right) \left(1-\frac{z}{q_3^2 w}\right) \left(1-\frac{q_2 z}{w}\right)
\left(1-\frac{z}{q_2w}\right)\cdot q_3^2w^4:M_3(q_3;z)M_3(q_1;w):.\nonumber 
\end{align}
Note that $M_4(q_1;z)=M_4(q_3;z)=0$. 
These operator products yield 
\begin{align}
&M_k(q_1;z)M_{\ell}(q_3;w)-M_{\ell}(q_3;w)M_k(q_1;z)=0\qquad (\forall k,\ell \in \{1,2\}), \\
&M_1(q_1;z)M_3(q_3;w)-M_3(q_3;w)M_1(q_1;z)
=-(1-q_1^2) \delta \left( \frac{w}{z} \right) :M_1(q_1;w)M_3(q_3;w):,\\
&M_2(q_1;z)M_3(q_3;w)-M_3(q_3;w)M_2(q_1;z)
=-(1-q_1^{-2}) \delta \left( \frac{q_3 w}{q_1z} \right):M_2(q_1;z)M_3(q_3;w):,\\
&M_3(q_1;z)M_1(q_3;w)-M_1(q_3;w)M_3(q_1;z)
=(1-q_3^2)\delta \left( \frac{w}{z} \right) :M_3(q_1;w)M_1(q_3;w):,\\
&M_3(q_1;z)M_2(q_3;w)-M_2(q_3;w)M_3(q_1;z)
=(1-q_3^{-2}) \delta \left( \frac{q_3 w}{q_1z} \right):M_3(q_1;z)M_2(q_3;w):,\\
&M_3(q_1;z)M_3(q_3;w)-M_3(q_3;w)M_3(q_1;z)=0.
\end{align}
Noting that 
\begin{align}
&(1-q_1^2)  :M_1(q_1;w)M_3(q_3;w):
=(1-q_3^2):M_3(q_1;w)M_1(q_3;w):,\\
&(1-q_1^{-2}) \delta \left( \frac{q_3 w}{q_1z} \right):M_2(q_1;z)M_3(q_3;w):w^2
=(1-q_3^{-2}) \delta \left( \frac{q_3 w}{q_1z} \right):M_3(q_1;z)M_2(q_3;w):z^2,
\end{align}
we obtain 
\begin{align}
[\mathcal{T}(q_1;z), \mathcal{T}(q_3;w)]=0.
\end{align}
\end{proof}

If $\xi=q_1^{\pm 1}$ or $\xi=q_3^{\pm 1}$, the operator $\mathcal{T}(\xi;z)$
satisfy the quadratic relation of the $q$-deformed Virasoro algebra \cite{SKAO1995quantum}. 

\begin{theorem}\label{thm: rel TT}
We obtain 
\begin{gather}
g^{(1)}\left( \frac{w}{z} \right) \mathcal{T}(q_1^{\pm 1};z)\mathcal{T}(q_1^{\pm 1};w)
-g^{(1)}\left( \frac{z}{w} \right) \mathcal{T}(q_1^{\pm 1};w)\mathcal{T}(q_1^{\pm 1};z)\nonumber  \\
=-\frac{(1-q_1^{2})(1-q_3/q_1)}{1-q_2^{-1}} \left(  \delta \left( \frac{w}{q_2z}\right) - \delta \left( \frac{q_2w}{ z}\right) \right), \label{eq: rel TT q1}\\
g^{(3)}\left( \frac{w}{z} \right) \mathcal{T}(q_3^{\pm 1};z)\mathcal{T}(q_3^{\pm 1};w)
-g^{(3)}\left( \frac{z}{w} \right) \mathcal{T}(q_3^{\pm 1};w)\mathcal{T}(q_3^{\pm 1};z)\nonumber \\
=-\frac{(1-q_3^{2})(1-q_1/q_3)}{1-q_2^{-1}} \left(  \delta \left( \frac{w}{q_2z}\right) - \delta \left( \frac{q_2w}{ z}\right) \right). \label{eq: rel TT q3}
\end{gather}
Here, the structure functions $g^{(1)}\left( z\right)$ and $g^{(3)}\left( z\right)$ are defined by 
\begin{align}
g^{(1)}\left( z\right)=\exp \left( \sum_{n>0} \frac{(1-q_1^{2n})(1-q_3^n/q_1^n)}{n(1+q_2^{-n})} z^n \right), \\
g^{(3)}\left( z\right)=\exp \left( \sum_{n>0} \frac{(1-q_3^{2n})(1-q_1^n/q_3^n)}{n(1+q_2^{-n})} z^n \right).  
\end{align}
\end{theorem}

While the $q$-deformed Virasoro algebra is typically realized by two vertex operators, 
the operator $\mathcal{T}(q_k^{\pm 1};z)$ ($k=1,3$) 
is constructed from three vertex operators $M_i(q_k^{\pm 1};z)$ ($i=1,2,3$ or $i=1,2,4$). 
Hence, it is rather nontrivial that they satisfy the relation of the $q$-deformed Virasoro algebra.
We also note that the structure functions $g^{(1)}(z)$ and $g^{(3)}(z)$ 
exhibit the same parameter dependence as two commutative $\mathcal{E}_1$'s 
embedded in $\mathcal{E}_2$ \cite{FJMM2016branching,FJM2021Evaluation}. 
In fact,   $\mathcal{T}(q_k;z)$  ($k=1,3$)  correspond to the fused currents of  \cite{FJMM2016branching};  
in our notation, they can be given by
\begin{align}
&\lim_{w \to z}  \bigg(1-\frac{z}{w}\bigg)  \bigg(1-\frac{q_3 z}{q_1w} \bigg) X_1(q_1w) X_2(z), \\
&\lim_{w \to z}  \bigg(1-\frac{z}{w} \bigg)  \bigg(1-\frac{q_1 z}{q_3w} \bigg) X_1(q_3w) X_2(z)
%\lim_{w \to z}  \Big(1-\frac{z}{w}\Big)  \Big(1-\frac{z}{q_k^2 \, q_2w} \Big) X_1(q_kw) X_2(z), 
\end{align}
after decoupling Cartan components.

\begin{proof}[Proof of Theorem \ref{thm: rel TT}]
%We prove only (\ref{eq: rel TT q1}). The relation (\ref{eq: rel TT q3}) can be shown in a similar manner. 
We first note that the sum of the first two terms in $\mathcal{T}(q_k^{\pm 1};z)$ ($k=1,3$), 
namely $M_1(q_k^{\pm 1};z)+M_2(q_k^{\pm 1};z)$, 
satisfies the quadratic relation of the $q$-deformed Virasoro algebra. That is,
\begin{align}
&g^{(k)}\left( \frac{w}{z} \right) \left\{ M_1(q_k^{\pm 1};z)+M_2(q_k^{\pm 1};z)\right\} 
\left\{ M_1(q_k^{\pm 1};w)+M_2(q_k^{\pm 1};w) \right\}\nonumber \\
&\quad -g^{(k)}\left( \frac{z}{w} \right) \left\{ M_1(q_k^{\pm 1};w)+M_2(q_k^{\pm 1};w) \right\}
\left\{ M_1(q_k^{\pm 1};z)+M_2(q_k^{\pm 1};z) \right\}  \\
&=\mathcal{C}^{(k)}\cdot  \left(  \delta \left( \frac{w}{q_2z}\right) - \delta \left( \frac{q_2w}{ z}\right) \right), \qquad \mathcal{C}^{(k)}\equiv -\frac{(1-q_k^2)(1-q_2^{-1}q_k^{-2})}{1-q_2^{-1}}. \nonumber
\end{align}
This relation can be proved by the same argument 
as in the standard free field realization of the $q$-deformed Virasoro algebra. 
As for the terms involving $M_3(q_k;z)$ or $M_4(q_{k}^{-1};z)$, 
we can proceed in the usual way. Let $i_{+}=3$ and $i_{-}=4$. 
We have 
\begin{align}
&g^{(k)}\left( \frac{w}{z} \right)M_{i_{\pm}}(q_k^{\pm 1};z)M_{i_{\pm}}(q_k^{\pm 1};w)-
g^{(k)}\left( \frac{z}{w} \right)M_{i_{\pm}}(q_k^{\pm 1};w)M_{i_{\pm}}(q_k^{\pm 1};z)=0, \\ %next M1 M_{ipm}
&g^{(k)}\left( \frac{w}{z} \right)M_{1}(q_k^{\pm 1};z)M_{i_{\pm}}(q_k^{\pm 1};w)-
g^{(k)}\left( \frac{z}{w} \right)M_{i_{\pm}}(q_k^{\pm 1};w)M_{1}(q_k^{\pm 1};z)\nonumber \\
&\qquad \qquad \quad   =-(1-q_k^2) \delta \left( \frac{w}{z} \right) 
:M_{1}(q_k^{\pm 1};w)M_{i_{\pm}}(q_k^{\pm 1};w):, \\
&g^{(k)}\left( \frac{w}{z} \right)M_{i_{\pm}}(q_k^{\pm 1};z)M_{1}(q_k^{\pm 1};w)-
g^{(k)}\left( \frac{z}{w} \right)M_{1}(q_k^{\pm 1};w)M_{i_{\pm}}(q_k^{\pm 1};z)\nonumber\\
&\qquad \qquad \quad =(1-q_k^2) \delta \left( \frac{w}{z} \right) 
:M_{1}(q_k^{\pm 1};w)M_{i_{\pm}}(q_k^{\pm 1};w):, \\  %next M2 M_{ipm}
&g^{(k)}\left( \frac{w}{z} \right)M_{2}(q_k^{\pm 1};z)M_{i_{\pm}}(q_k^{\pm 1};w)-
g^{(k)}\left( \frac{z}{w} \right)M_{i_{\pm}}(q_k^{\pm 1};w)M_{2}(q_k^{\pm 1};z)\nonumber\\
&\qquad \qquad \quad =-(1-q_k^{-2}) \delta \left( \frac{w}{z} \right) 
:M_{2}(q_k^{\pm 1};w)M_{i_{\pm}}(q_k^{\pm 1};w):, \\
&g^{(k)}\left( \frac{w}{z} \right)M_{i_{\pm}}(q_k^{\pm 1};z)M_{2}(q_k^{\pm 1};w)-
g^{(k)}\left( \frac{z}{w} \right)M_{2}(q_k^{\pm 1};w)M_{i_{\pm}}(q_k^{\pm 1};z)\nonumber\\
&\qquad \qquad \quad =(1-q_k^{-2}) \delta \left( \frac{w}{z} \right) 
:M_{2}(q_k^{\pm 1};w)M_{i_{\pm}}(q_k^{\pm 1};w):. 
\end{align}
By adding these relations, we obtain the theorem. 
\end{proof}

The quadratic relations between $W_i(z)$ and $\mathcal{T}(q_k^{\pm 1};w)$ are given as follows. 
%Let us omit the proof since it is straightforward. 

\begin{theorem}For $k=1,3$, we obtain 
\begin{gather}
q_k^{\pm 1} \cdot f\left( q_k^{\pm 1} ; \frac{w}{z} \right) W_1(z) \mathcal{T}(q_k^{\pm 1};w) 
- f\left( q_k^{\mp 1};\frac{z}{w} \right) \mathcal{T}(q_k^{\pm 1} ;w)  W_1(z)\nonumber \\
=\pm q^{\pm 1} (q_k-q_k^{-1}) \delta \bigg( \frac{q_k^{\pm 1} w}{q_2^{\pm 1}z} \bigg) W_2(w),\\
q_k^{\mp 1} \cdot f\Big( q_k^{\mp 1} ; \frac{q_k^{\pm 1} w}{z} \Big) W_2(z) \mathcal{T}(q_k^{\pm 1};w) 
- f\Big( q_k^{\pm 1};\frac{z}{q_k^{\pm 1} w} \Big) \mathcal{T}(q_k^{\pm 1};w)  W_2(z)\nonumber\\
=\mp q^{\mp 1} (q_k-q_k^{-1})\delta \left( \frac{q_2^{\pm 1}w}{z} \right)W_1(q_k^{\pm 1} w).
\end{gather}
\begin{comment}
\begin{gather} %only T(q_k^+z)
q_k \cdot f\left( q_k ; \frac{w}{z} \right) W_1(z) \mathcal{T}(q_k;w) 
- f\left( q_k^{-1};\frac{z}{w} \right) \mathcal{T}(q_k;w)  W_1(z)\nonumber \\
=q(q_k-q_k^{-1})\delta \left( \frac{q_kw}{q_2z} \right)W_2(w),\\
q_k^{-1} \cdot f\left( q_k^{-1} ; \frac{q_kw}{z} \right) W_2(z) \mathcal{T}(q_k;w) 
- f\Big( q_k;\frac{z}{q_kw} \Big) \mathcal{T}(q_k;w)  W_2(z)\nonumber\\
=-q^{-1} (q_k-q_k^{-1})\delta \left( \frac{q_2w}{z} \right)W_1(q_k w).
\end{gather}
\end{comment}
\end{theorem}

\begin{proof}
We prove only the relation between $W_1(z)$ and $\mathcal{T}(q_k;w)$. 
The other relations can be shown similarly. 
A direct calculation yields 
\begin{align}
&q_k \cdot f\left( q_k ; \frac{w}{z} \right) \Lambda^+_1(z) M_1(q_k;w) -
f\left( q_k^{-1} ; \frac{z}{w} \right) M_1(q_k;w)\Lambda^+_1(z)=0, \\
&q_k \cdot f\left( q_k ; \frac{w}{z} \right) \Lambda^-_1(z) M_2(q_k;w) -
f\left( q_k^{-1} ; \frac{z}{w} \right) M_2(q_k;w)\Lambda^-_1(z)=0, \\
&q_k \cdot f\left( q_k ; \frac{w}{z} \right) \Lambda^+_1(z) M_2(q_k;w) -
f\left( q_k^{-1} ; \frac{z}{w} \right) M_2(q_k;w)\Lambda^+_1(z)\nonumber \\
&\qquad \qquad\quad =(q_k-q_k^{-1})\delta \Big( \frac{q_kw}{z} \Big) :\Lambda^+_1(z) M_2(q_k;w):, \\
&q_k \cdot f\left( q_k ; \frac{w}{z} \right) \Lambda^-_1(z) M_1(q_k;w) -
f\left( q_k^{-1} ; \frac{z}{w} \right) M_1(q_k;w)\Lambda^-_1(z)\nonumber \\
&\qquad \qquad\quad =(q_k-q_k^{-1})\delta \Big( \frac{q_kw}{q_2z} \Big) :\Lambda^-_1(z) M_1(q_k;w):\\
&q_k \cdot f\left( q_k ; \frac{w}{z} \right) \Lambda^+_1(z) M_3(q_k;w) -
f\left( q_k^{-1} ; \frac{z}{w} \right) M_3(q_k;w)\Lambda^+_1(z)=0, \\
&q_k \cdot f\left( q_k ; \frac{w}{z} \right) \Lambda^-_1(z) M_3(q_k;w) -
f\left( q_k^{-1} ; \frac{z}{w} \right) M_3(q_k;w)\Lambda^-_1(z) \nonumber \\
&\qquad \qquad\quad = \frac{q_2^{-1}z^{-1}w^{-1}}{1-q_2^{-1}} \left( \delta \Big( \frac{q_kw}{z}\Big) 
-\delta \Big( \frac{q_kw}{q_2z} \Big)\right):\Lambda^-_1(z) M_3(q_k;w):.
\end{align}
Noting that 
\begin{align}
&(q_k-q_k^{-1})\delta \Big( \frac{q_kw}{z} \Big) :\Lambda^+_1(z) M_2(q_k;w):+
\frac{q_2^{-1}z^{-1}w^{-1}}{1-q_2^{-1}}  \delta \Big( \frac{q_kw}{z}\Big) :\Lambda^-_1(z) M_3(q_k;w):w^2=0,\\
&(q_k-q_k^{-1}) \delta \Big( \frac{q_kw}{q_2z} \Big):\Lambda^-_1(z) M_1(q_k;w):
=q(q_k-q_k^{-1}) \delta \Big( \frac{q_kw}{q_2z} \Big)\Lambda^+_2(w), \\
&-\frac{q_2^{-1}z^{-1}w^{-1}}{1-q_2^{-1}} \delta \Big( \frac{q_kw}{q_2z} \Big) 
:\Lambda^-_1(z) M_3(q_k;w):w^2=
q(q_k-q_k^{-1}) \delta \Big( \frac{q_kw}{q_2z} \Big) :\Lambda^-_2(w):, 
\end{align}
we have 
\begin{gather}
q_k \cdot f\left( q_k ; \frac{w}{z} \right) W_1(z) \mathcal{T}(q_k;w) 
- f\left( q_k^{-1};\frac{z}{w} \right) \mathcal{T}(q_k;w)  W_1(z)\nonumber \\
=q(q_k-q_k^{-1})\delta \left( \frac{q_kw}{q_2z} \right)W_2(w).
\end{gather}
\end{proof}

%The similar relations hold for $\mathcal{T}(q_1^{-1};w)$ and $\mathcal{T}(q_3^{-1};w)$. We shall omit their details. 
The degenerate limit of the operator $\mathcal{T}(\xi;z)$ takes the following form. 
From this limit, we can obtain the generator $T(z)$ of \svir.

\begin{proposition}\label{prop: limit calT}
Let $x$ be a complex parameter independent of $\hbar$, and let $\xi =e^{x\hbar}$. 
Then the $\hbar$-expansion of $\mathcal{T}(\xi;z)$ is of the form  
\begin{align}
&\mathcal{T}(\xi;z)=2+\left( 2\beta \, \mathsf{T}(x;z)\,  z^2 +\frac{(\beta-1)^2}{4} \right) \hbar^2 +O(\hbar^3),\\
&\mathsf{T}(x;z)\equiv T(z) 
+\left( \frac{x^2}{\beta} -\frac{1}{2}\right) \NPb \widetilde{\psi}'(z) \widetilde{\psi}(z) \NPb 
+\frac{x}{\sqrt{-\beta}}\, \widetilde{\psi}(z)\, G(z), \label{eq: sfT}
\end{align}
where $T(z)$ and $G(z)$ are the operators introduced in Definition \ref{def: SCA rep}, 
with $\di \sigma=\frac{1-\beta}{2\sqrt{\beta}}$. 
\end{proposition}

\begin{remark}\label{rem: commuting vir}
If $(x_1,x_2)=(\pm 1,\mp\beta)$,  
the operators $\mathsf{T}(x_1;z)$ and $\mathsf{T}(x_2;z)$ correspond, up to scalar multiples, 
to two commuting Virasoro algebras discussed in \cite{BBFLT2011Instanton}.\footnote{
The parameter $\beta$ corresponds to $-b^2$ or $-b^{-2}$ in \cite{BBFLT2011Instanton}.}
These commuting Virasoro algebras are based on a certain factorization property 
of coset conformal field theories 
(See also \cite{Wyllard2011Coset,CSS1989Renormalization,CPSS1989Fusions,Lashkevich1992Superconf}). 
\end{remark}

\begin{proof}[Proof of Proposition \ref{prop: limit calT}]
We set 
\begin{align}
&\overline{M}_1(\xi;z)=\varphi_q\left(\xi q^{-1} z\right)+\widetilde{\varphi}_q\left(\xi q^{-1} z\right)
-\varphi_q\left(q^{-1} z\right)+\widetilde{\varphi}_q\left(q^{-1} z\right)+ \frac{\beta - 1}{2}\hbar, \\
&\overline{M}_2(\xi;z)=-\varphi_q\left(\xi q z\right)-\widetilde{\varphi}_q\left(\xi q z\right)
+\varphi_q\left(q z\right)-\widetilde{\varphi}_q\left(q z\right)- \frac{\beta-1}{2}\hbar, 
\end{align}
so that $M_1(\xi;z)=e^{\overline{M}_1(\xi;z)},M_2(\xi;z)=e^{\overline{M}_2(\xi;z)}$. 
Since $\di \lim_{\hbar \to 0}\overline{M}_1(\xi;z)=\lim_{\hbar \to 0}\overline{M}_2(\xi;z)=0$, 
we have 
\begin{align}
\lim_{\hbar \rightarrow 0} M_1(\xi; z)=\lim_{\hbar \rightarrow 0} M_2(\xi; z)=1. 
\end{align}
This implies $\di \mathrm{Coeff}_{\hbar^0}\mathcal{T}(\xi;z)= 2$. 
Note that the $\hbar$-expansions of $M_3(\xi;z)$ and $M_4(\xi;z)$ contain 
no terms of order lower than $\hbar^2$. 

Next, consider the first derivatives with respect to $\hbar$: 
\begin{align}
%&\lim _{\hbar \rightarrow 0} \frac{\partial}{\partial \hbar}\overline{M}_1(\xi ; z)
%=x\left(\sum_{n \in \mathbb{Z}} a_n z^{-n}\right)
%+\sqrt{\beta}\left(\sum_{n \in \mathbb{Z}} \widetilde{a}_n z^{-n}\right)+\frac{\beta-1}{2},\\
%&\lim _{\hbar \rightarrow 0} \frac{\partial}{\partial \hbar}\overline{M}_2(\xi ; z)
%=-x\left(\sum_{n \in \mathbb{Z}} a_n z^{-n}\right)
%-\sqrt{\beta}\left(\sum_{n \in \mathbb{Z}} \widetilde{a}_n z^{-n}\right)-\frac{\beta-1}{2},
&\lim _{\hbar \rightarrow 0} \frac{\partial}{\partial \hbar}\overline{M}_1(\xi ; z)
=x \,\phi'(z)\cdot z
+\sqrt{\beta} \,\widetilde{\phi}'(z)\cdot z+\frac{\beta-1}{2},\\
&\lim _{\hbar \rightarrow 0} \frac{\partial}{\partial \hbar}\overline{M}_2(\xi ; z)
=-x \, \widetilde{\phi}'(z)\cdot z
-\sqrt{\beta}\,\widetilde{\phi}'(z)\cdot z -\frac{\beta-1}{2}. 
\end{align}
Thus, $\di \mathrm{Coeff}_{\hbar^1}\mathcal{T}(\xi;z)= 0$.

The second derivatives yield
\begin{align}
&\lim_{\hbar \rightarrow 0} \frac{\partial^2}{\partial \hbar^2} \overline{M}_1(\xi;z)
=\left(-x^2-x+\beta x\right)\Bigg(\sum_{n \neq 0} n a_n z^{-n}\Bigg)
+\left(\beta^{\frac{3}{2}}-x \sqrt{\beta}\right)\Bigg(\sum_{n \neq 0} n \widetilde{a}_n z^{-n}\Bigg),\\
&\lim_{\hbar \rightarrow 0} \frac{\partial^2}{\partial \hbar^2} \overline{M}_2(\xi;z)
=\left(x^2-x+\beta x\right)\Bigg(\sum_{n \neq 0} n a_n z^{-n}\Bigg)
+\left(-2\sqrt{\beta}+\beta^{\frac{3}{2}}+x \sqrt{\beta}\right)\Bigg(\sum_{n \neq 0} n \widetilde{a}_n z^{-n}\Bigg).
\end{align}
Adding these, we have 
\begin{align}
\lim_{\hbar \rightarrow 0} \frac{\partial^2}{\partial \hbar^2} 
\left( \overline{M}_1(\xi;z)+\overline{M}_2(\xi;z) \right)
=2(\beta -1) \Bigg\{ x \Bigg(\sum_{n \neq 0} n a_n z^{-n}\Bigg) 
+\sqrt{\beta} \Bigg(\sum_{n \neq 0} n \widetilde{a}_n z^{-n}\Bigg) \Bigg\}.
\end{align}
Note that the second derivatives of $M_i(\xi; z)$ for $i=1,2$ can be computed as 
\begin{align}
\lim_{\hbar \rightarrow 0}  \frac{\partial^2}{\partial \hbar^2} M_i(\xi;z)
=\lim_{\hbar \rightarrow 0}\left\{ \left( \frac{\partial}{\partial \hbar}\overline{M}_i(\xi ; z) \right)^2
+ \frac{\partial^2}{\partial \hbar^2} \overline{M}_i(\xi;z) \right\} \quad (i=1,2). 
\end{align}
Furthermore, the second derivatives of $M_3(\xi; z)$ and $M_4(\xi; z)$ yield
\begin{align}
&\lim_{\hbar \rightarrow 0}  \frac{\partial^2}{\partial \hbar^2} M_3(\xi;z)
=2(1+x)(x-\beta):e^{2 \phi(z)}:,\\
&\lim_{\hbar \rightarrow 0}  \frac{\partial^2}{\partial \hbar^2} M_4(\xi;z)
=2(-1+x)(x+\beta):e^{-2 \phi(z)}:. 
\end{align}
Combining all terms, we obtain 
\begin{align}\label{eq: sum lim M1M2}
\lim_{\hbar \rightarrow 0}  \frac{\partial^2}{\partial \hbar^2} \mathcal{T}(\xi;z) 
=\Bigg\{ & 2x^2 \phi'(z)^2 + 2\beta \, \widetilde{\phi}'(z)^2 +4x\sqrt{\beta} \, \phi'(z)\widetilde{\phi}'(z)
-2(\beta -1)\left( x \phi''(z) +\sqrt{\beta}\, \widetilde{\phi}''(z) \right) \nonumber  \\
&+2(1+x)(x-\beta):e^{2 \phi(z)}:+2(-1+x)(x+\beta):e^{-2 \phi(z)}:\Bigg\}\cdot z^2+\frac{(\beta-1)^2}{2}
\end{align}
Finally, by applying Lemma \ref{lem: fml BF corresp} (Appendix \ref{sec: fml bf corresp}) to the right hand side of (\ref{eq: sfT}), 
we obtain
\begin{align}
\mathrm{Coeff}_{\hbar^2}\mathcal{T}(\xi;z) = 2\beta \, \mathsf{T}(x;z)\,  z^2 +\frac{(\beta-1)^2}{4}. 
\end{align}
\end{proof}

The following combination allows us to extract only 
$T(z)$ from the limit.

\begin{corollary}
Let $k_1$ and $k_2$ be non-zero complex parameters which are independent of $\hbar$, 
and set 
\begin{align}
\xi_1= \exp \left( \pm\sqrt{\dfrac{\beta k_2}{2 k_1}} \hbar \right), \quad 
\xi_2= \exp \left( \mp \sqrt{\dfrac{\beta k_1}{2 k_2}} \hbar \right). \label{eq: xik1k2}
\end{align} 
Then, we have 
\begin{align}
k_1 \mathcal{T}(\xi_1;z)+k_2 \mathcal{T}(\xi_2;z)=2(k_1+k_2)+(k_1+k_2)\left( 2 \beta \, T(z) \cdot z^2 +\frac{1}{4}(\beta -1)^2 \right)\hbar^2+O(\hbar^3).
\end{align}
\end{corollary}

This  follows immediately from Proposition \ref{prop: limit calT}.

\begin{remark}
In this paper, we define the highest weight vectors $\ket{n,u}$ and the Fock module $\mathcal{F}_u$ 
so that the action of the zero modes $\boo_{i,0}$ yields integers, 
and we restrict our discussion to the Neveu--Schwarz sector. 
By modifying the Fock module and the highest weight vectors 
so that the action of $\boo_{i,0}$ gives half-integer values, 
we should also be able to treat the algebra and its representations in the Ramond sector.
\end{remark}

\section*{Acknowledgment}

The authors would like to thank H. Awata, H. Kanno and J. Shiraishi 
for valuable comments. 
This work is partially supported by JSPS KAKENHI 
(21K13803).

\appendix

\section{Proof of Proposition \ref{prop: rep tor gl2}}\label{sec: pf rep}

In this appendix, we give proof of Proposition \ref{prop: rep tor gl2}. 
By definition, it is clear that the relations (\ref{eq: def rel KK c}) and (\ref{eq: def rel KK}) hold 
on the representation. 
The operator products among the vertex operators take the form 
\begin{align}
&\rho_u (K^+_i(z))\cdot\rho_u (K^-_j(w))\\ 
&=\begin{cases}
\di \frac{(1-q^{-3}w/z)(1-q^3w/z)}{(1-q^{-1}w/z)(1-qw/z)}:\mathrm{(l.h.s.)}:& (i=j),\\ \\
\di \frac{(1-q_1 q w/z)^2 (1-q q_3 w/z)^2}{(1-q_1q^{-1}w/z)(1-q_3q^{-1}w/z)(1-qq_1^{-1}w/z)(1-qq_3^{-1}w/z)}
:\mathrm{(l.h.s.)}:& (i\neq j),\\
\end{cases}\nonumber \\
&\rho_u (K^-_i(z))\cdot\rho_u (K^+_j(w)) = :\mathrm{(l.h.s.)}: \quad (\forall i,j ),\\
&\rho_u (K^{\pm}_i(z))\cdot\rho_u (K^{\pm}_j(w)) = :\mathrm{(l.h.s.)}: \quad (\forall i,j ),
\end{align}
\begin{align}
&\rho_u (K^+_i(z))\cdot\rho_u (E_j(w))=
\begin{cases}
\di \frac{1-q_2^{-1}w/z}{1-q_2 w/z}\cdot q_2 :\mathrm{(l.h.s.)}: & (i= j),\\
\di \frac{(1-q_1^{-1}w/z)(1-q_3^{-1}w/z)}{(1-q_1w/z)(1-q_3w/z)}\cdot q_2^{-1} :\mathrm{(l.h.s.)}: & (i\neq j),
\end{cases}\\
&\rho_u (K^-_i(z))\cdot\rho_u (E_j(w))=
\begin{cases}
\di  q_2^{-1} :\mathrm{(l.h.s.)}: & (i= j),\\
\di  q_2 :\mathrm{(l.h.s.)}: & (i\neq j),
\end{cases}\\ 
&\rho_u (E_i(z))\cdot \rho_u (K^+_j(w))= :\mathrm{(l.h.s.)}: \quad (\forall i,j ),\\
&\rho_u (E_i(z))\cdot \rho_u (K^-_j(w))=
\begin{cases}
\di \frac{1-q^{-3}w/z}{1-q w/z} :\mathrm{(l.h.s.)}: & (i= j),\\
\di \frac{(1-q_1 q w/z)(1-q \, q_3w/z)}{(1-q_1 q^{-1} w/z)(1-q_3 q^{-1} w/z)} :\mathrm{(l.h.s.)}: & (i\neq j),
\end{cases}
\end{align}
\begin{align}
&\rho_u (K^+_i(z))\cdot\rho_u (F_j(w))=
\begin{cases}
\di \frac{1-q^3w/z}{1-q^{-1} w/z}\cdot q_2^{-1} :\mathrm{(l.h.s.)}: & (i= j),\\
\di \frac{(1-q_1 q w/z)(1-q q_3w/z)}{(1-q_1 q^3 w/z)(1-q^3 q_3 w/z)}\cdot q_2 :\mathrm{(l.h.s.)}: & (i\neq j),
\end{cases}\\
&\rho_u (K^-_i(z))\cdot\rho_u (F_j(w))=
\begin{cases}
\di  q_2 :\mathrm{(l.h.s.)}: & (i= j),\\
\di  q_2^{-1} :\mathrm{(l.h.s.)}: & (i\neq j),
\end{cases}\\ 
&\rho_u (F_i(z))\cdot \rho_u (K^+_j(w))= :\mathrm{(l.h.s.)}: \quad (\forall i,j ),\\
&\rho_u (F_i(z))\cdot \rho_u (K^-_j(w))=
\begin{cases}
\di \frac{1-q_2 w/z}{1-q_2^{-1} w/z} :\mathrm{(l.h.s.)}: & (i= j),\\
\di \frac{(1-q_1 w/z)(1- q_3w/z)}{(1-q_1^{-1} w/z)(1-q_3^{-1} w/z)} :\mathrm{(l.h.s.)}: & (i\neq j),
\end{cases}
\end{align}
\begin{align}
&\rho_u (E_i(z))\cdot\rho_u (E_j(w))=
\begin{cases}
\di (1-w/z)(1-q_2^{-1}w/z) z^2 :\mathrm{(l.h.s.)}: & (i=j),\\
\di  \frac{z^{-2}}{(1-q_1 w/z)(1-q_3w/z)} :\mathrm{(l.h.s.)}: & (i\neq j),
\end{cases}\label{eq: contraction EE}\\
&\rho_u (F_i(z))\cdot\rho_u (F_j(w))=
\begin{cases}
\di (1-w/z)(1-q_2w/z) z^2 :\mathrm{(l.h.s.)}: & (i=j),\\
\di \frac{z^{-2}}{(1-q_1^{-1} w/z)(1-q_3^{-1}w/z)} :\mathrm{(l.h.s.)}: & (i\neq j),\\
\end{cases}
\end{align}
\begin{align}
&\rho_u (E_i(z))\cdot \rho_u (F_j(w))=
\begin{cases}
\di \frac{z^{-2}}{(1-qw/z)(1-q^{-1}w/z)} :\mathrm{(l.h.s.)}: & (i=j),\\
\di (1-qq_1w/z)(1-qq_3w/z)z^2 :\mathrm{(l.h.s.)}: & (i\neq j), 
\end{cases} \label{eq: contraction EF}\\
&\rho_u (F_i(z))\cdot \rho_u (E_j(w))=
\begin{cases}
\di \frac{z^{-2}}{(1-qw/z)(1-q^{-1}w/z)} :\mathrm{(l.h.s.)}: & (i=j),\\ 
\di (1-qq_1w/z)(1-qq_3w/z)z^2:\mathrm{(l.h.s.)}: & (i\neq j).
\end{cases}\label{eq: contraction FE}
\end{align}
Here, $:\mathrm{(l.h.s.)}:$ stands for the normal ordering of the left hand side. 
By these formulas for the operator products, 
the relations (\ref{eq: def rel KK})--(\ref{eq: def rel FF}) immediately follow. 

Next, we show the relation (\ref{eq: def rel EF}). 
If $i = j$, 
(\ref{eq: contraction EF}) and (\ref{eq: contraction FE}) % and the identity 
%\begin{align}
%\frac{1}{(1-x)(1-p x)}-\frac{p^{-1} x^{-2}}{(1-p^{-1} x^{-1})(1-x^{-1})}
%=\frac{1}{1-p} \left( \delta(x) -p \delta (px)\right)\label{eq: formal series formula}
%\end{align}
yield 
\begin{align}
&\rho_u (E_i(z))\cdot \rho_u (F_j(w))-\rho_u (F_j(w))\cdot \rho_u (E_i(z))\nonumber \\
%&=\frac{z^{-2}}{1-q_2^{-1}} \left( \delta(qw/z) - q_2^{-1} \delta(q^{-1}w/z) \right)z^{\bof_{i,0}+1}w^{-\bof_{i,0}+1} :\eta_i(z)\xi_i(w): \nonumber \\
&=\frac{1}{q-q^{-1}} \left( \delta(qw/z) : \eta_i(z)\xi_i(q^{-1}z) : q^{\bof_{i,0}}
 - \delta(q^{-1}w/z):\eta_i(z)\xi_i(qz): q^{-\bof_{i,0}}\right). 
\end{align}
By the relations 
\begin{align}
:\eta_i(z)\xi_i(q^{-1}z) : q^{\bof_{i,0}}= \varphi^+_i(q^{-\frac{1}{2}}z)\times q^{\bof_{i,0}}=\rho_u(K_i^{+}(z)), \\
:\eta_i(z)\xi_i(qz): q^{-\bof_{i,0}}=\varphi^-_i(q^{-\frac{1}{2}}z)\times q^{-\bof_{i,0}}=\rho_u(K_i^{-}(z)), 
\end{align}
we can show that (\ref{eq: def rel EF})  holds in the case of $i = j$. 
If $i \neq j$, (\ref{eq: contraction EF}) and (\ref{eq: contraction FE}) yield 
\begin{align}
\rho_u (E_i(z))\cdot \rho_u (F_j(w))-\rho_u (F_j(w))\cdot \rho_u (E_i(z))=0. 
\end{align}
Therefore,  (\ref{eq: def rel EF}) holds in the case $i\neq j$.

Next, we show the Serre relation (\ref{eq:serr1}). 
We define 
\begin{align}
&P_{i,j}(z,w)=\begin{cases}
(1 - w/z) (1 - q_2^{-1} w/z) z^2, & i=j,\\
\dfrac{1}{(1 - q_1 w/z) (1 - q_3 w/z) z^2}, & i \neq j,
\end{cases}
\label{eq: EiEj OPE}
\end{align}
which is the function appearing in the operator product (\ref{eq: contraction EE}). 
We then set 
\begin{align}
&P_{i_1,i_2,i_3,i_4}(z_1, z_2,z_3,z_4)=\prod_{1\leq k < \ell \leq 4} P_{i_k, i_{\ell}}(z_{k}, z_{\ell}).
\end{align}
With this notation, the left hand side of the Serre relation (\ref{eq:serr1}) can be written, under the representation, 
as 
\begin{align}
& \underset{z_1, z_2, z_3}{\operatorname{Sym}}
\left[\rho_u (E_i\left(z_1\right) )   ,\left[\rho_u ( E_i\left(z_2\right)) ,\left[\rho_u (E_i\left(z_3\right)), \rho_u (E_j(w))\right]_{q_2}\right]\right]_{q_2^{-1}}\nonumber \\
&\qquad =\underset{z_1, z_2, z_3}{\operatorname{Sym}} \mathcal{P}(z_1, z_2,z_3,w) : \prod_{k=1}^3\rho_u (E_i(z_k))\cdot \rho_u (E_j(w))  :   \qquad (i\neq j), \label{eq: rep serre}\\
&\mathcal{P}(z_1, z_2,z_3,w) =\Big\{
P_{1112}(z_1, z_2, z_3, w) - q_2 P_{1121}(z_1, z_2, w, z_3) -  P_{1121}(z_1, z_3, w, z_2) \nonumber \\ 
&\quad + q_2 P_{1211}(z_1, w, z_3, z_2) - q_2^{-1} \Big(P_{1121}(z_2, z_3, w, z_1) - q_2 P_{1211}(z_2, w, z_3, z_1) \nonumber \\ 
&\quad - P_{1211}(z_3, w, z_2, z_1) + q_2 P_{2111}(w, z_3, z_2, z_1) \Big) \Big\}. \label{eq: cal P}
\end{align}
Note that in the case $i\neq j$ in (\ref{eq: contraction EE}), 
the operator product converges when $|q_1|<1, |q_3|<1$ and $|z|=|w|$. 
Thus, the operator product in (\ref{eq: rep serre}) also converges in the same region even after exchanging $z_i$'s and $w$. 
Therefore, (\ref{eq: cal P}) can be computed as an ordinary rational function. 
A direct computation shows that 
\begin{align}
\underset{z_1, z_2, z_3}{\operatorname{Sym}} \mathcal{P}(z_1, z_2,z_3,w) =0, \label{eq: sym calP}
\end{align}
which implies the Serre relation (\ref{eq:serr1}). 
The Serre relation (\ref{eq:serr2}) can be shown in a similar manner.
Thus, Proposition \ref{prop: rep tor gl2} is proved. 
\qed

\section{Some formulas on the boson-fermion correspondence}\label{sec: fml bf corresp}

In this appendix, we present several formulas on the boson-fermion correspondence. 

\begin{lemma}\label{lem: fml BF corresp}
Under the correspondence (\ref{eq: realize psi}), it follows that 
\begin{align}
&\NPb \psi'(w) \psi(w) \NPb 
=\frac{1}{2}  \left( :\phi'(w)^2 -e^{2\phi(w)} -e^{-2\phi(w)} : \right) \label{eq: fml BF corrsp 1}, \\
&\NPb \widetilde{\psi}'(w) \widetilde{\psi}(w) \NPb 
=\frac{1}{2}  \left( :\phi'(w)^2 +e^{2\phi(w)} +e^{-2\phi(w)} : \right) , \label{eq: fml BF corrsp 2}\\
&\widetilde{\psi}(w) \psi(w)=\sqrt{-1} \phi'(w), \label{eq: fml BF corrsp 3}\\
&\widetilde{\psi}(w) \psi'(w)=\frac{1}{2\sqrt{-1}} \left(:-e^{2\phi(w)} +e^{-2\phi(w)} -\phi''(w): \right). 
\label{eq: fml BF corrsp 4}
\end{align}
\end{lemma}

\begin{proof}
By rewriting the normal ordering of fermions in terms of bosons, we have  
\begin{align}
\NPb \psi(z) \psi(w) \NPb = &\,  \psi(z) \psi(w)-\frac{1}{z-w} \nonumber \\
=&-\frac{1}{2}\Big( :(z-w)e^{\phi(z)+\phi(w)}-\frac{1}{z-w}e^{\phi(z)-\phi(w)} \\
&\qquad \quad -\frac{1}{z-w}e^{-\phi(z)+\phi(w)}+(z-w)e^{-\phi(z)-\phi(w)}: \Big)-\frac{1}{z-w}.\nonumber 
\end{align}
Taking the derivative with respect to $z$ and 
expanding around $w$ in Laurent series, we find
\begin{align}
\NPb \psi'(z) \psi(w) \NPb =\frac{1}{2} \left( :\phi'(w)^2 -e^{2\phi(w)} -e^{-2\phi(w)} :\right) +O(z-w). 
\end{align}
Taking the limit $z\to w$ yields (\ref{eq: fml BF corrsp 1}). 

Similarly, we have 
\begin{align}
\NPb\widetilde{\psi}(z) \widetilde{\psi}(w)\NPb
=&\, \widetilde{\psi}(z) \widetilde{\psi}(w)-\frac{1}{z-w} \nonumber \\
=&\frac{1}{2}\Big( :(z-w)e^{\phi(z)+\phi(w)}+\frac{1}{z-w}e^{\phi(z)-\phi(w)} \\
&\qquad \quad +\frac{1}{z-w}e^{-\phi(z)+\phi(w)}+(z-w)e^{-\phi(z)-\phi(w)}: \Big)-\frac{1}{z-w}\nonumber 
\end{align}
and 
\begin{align}
\NPb \widetilde{\psi}'(z) \widetilde{\psi}(w) \NPb =\frac{1}{2} \left( :\phi'(w)^2 +e^{2\phi(w)} +e^{-2\phi(w)} :\right) +O(z-w). 
\end{align}
Therefore, we obtain (\ref{eq: fml BF corrsp 2}). 

Moreover, we have 
\begin{align}
\widetilde{\psi}(z) \psi(w)=&\frac{1}{2\sqrt{-1}}
\Big( :(z-w)e^{\phi(z)+\phi(w)}-\frac{1}{z-w}e^{\phi(z)-\phi(w)} \nonumber \\
&\qquad \qquad +\frac{1}{z-w}e^{-\phi(z)+\phi(w)}-(z-w)e^{-\phi(z)-\phi(w)}: \Big) \\
=&\frac{1}{2\sqrt{-1}}\left\{ -2 \phi'(w) +\left(e^{2\phi(w)} -e^{-2\phi(w)} -\phi''(w)\right)(z-w) \right\}+O\left((z-w)^2\right)\nonumber
\end{align}
and 
\begin{align}
\widetilde{\psi}(z) \psi'(w)
=\frac{1}{2\sqrt{-1}}\left( -e^{2\phi(w)} +e^{-2\phi(w)} -\phi''(w)\right)+O(z-w). 
\end{align}
Thus, we can obtain (\ref{eq: fml BF corrsp 3}) and (\ref{eq: fml BF corrsp 4}).
\end{proof}

\section{Operator product formulas of the screening currents}\label{sec: OPE screening}

In this Appendix, we list formulas for the operator products among the screening currents: 
\begin{align}
&S^+_i(z)S^+_i(w)=(1-w/z)\frac{(q_2^{-1}w/z;q_3^2)_{\infty}}{(q_1^{-1}q_3w/z;q_3^2)_{\infty}}z^{1+\frac{1}{\beta}}
:S^+_i(z)S^+_i(w):,\\
&S^+_i(z)S^+_j(w)=\frac{(q_1q_3^2w/z;q_3^2)_{\infty}}{(q_1^{-1}w/z;q_3^2)_{\infty}}z^{-1+\frac{1}{\beta}}
:S^+_i(z)S^+_j(w):\qquad (i\neq j),
\end{align}
\begin{align}
&S^-_i(z)S^-_i(w)=(1-w/z)\frac{(q_2^{-1}w/z;q_1^2)_{\infty}}{(q_1q_3^{-1}w/z;q_1^2)_{\infty}}z^{1+\beta}
:S^-_i(z)S^-_i(w):,\\
&S^-_i(z)S^-_j(w)=\frac{(q_1^2q_3w/z;q_1^2)_{\infty}}{(q_3^{-1}w/z;q_1^2)_{\infty}}z^{-1+\beta}
:S^-_i(z)S^-_j(w):\qquad (i\neq j),
\end{align}
\begin{align}
&S^{\pm}_i(z)S^{\mp}_i(w)=:S^{\pm}_i(z)S^{\mp}_i(w):,\\
&S^{\mp}_i(z)S^{\pm}_i(w)=:S^{\mp}_i(z)S^{\pm}_i(w):,
\end{align}
\begin{align}
&S^{\pm}_i(z)S^{\mp}_j(w)=\frac{z^{-2}}{(1-qw/z)(1-q^{-1}w/z)}:S^{\pm}_i(z)S^{\mp}_j(w):,\\
&S^{\mp}_i(z)S^{\pm}_j(w)=\frac{z^{-2}}{(1-qw/z)(1-q^{-1}w/z)}:S^{\mp}_i(z)S^{\pm}_j(w):\qquad (i\neq j).
\end{align}
Here, we used the standard notation $\di (a;q)_{\infty}=\prod_{n=1}^{\infty}(1-q^{n-1}a)$. 
Note that the operator product formulas for the screening currents $S^{\pm}_i(z)$ 
are slightly different from the ones for bosonic screenings of the quantum affine algebra  $U_q(\widehat{\mathfrak{sl}}_2)$ or the $q$-deformed $N=2$ superconformal algebra \cite{Matsuo1994qdeformation,AHKS2024quantum}.

%%%%%%%%%%%%%%%%%%%%%%%%%%%%%%%%%%%%%
%%%%%%%%%%%%%%%%%%%%%%%%%%%%%%%%%%%%%
%%%%%%%%%%%%%%%%%%%%%%%%%%%%%%%%%%%%%
%\bibliographystyle{myutcaps}
%\bibliographystyle{myspmpsci}
%\bibliography{myreferences}

\begin{thebibliography}{10}

\bibitem{DI1997Generalization}
J.~Ding and K.~Iohara, ``{Generalization and deformation of Drinfeld quantum
  affine algebras},''
{\em Lett. Math. Phys.} {\bfseries 41} (1997) 181--193.
%%CITATION = LMPHD,41,181;%%.

\bibitem{Miki:2007}
K.~Miki, ``A {$(q,\gamma)$} analog of the {$W_{1+\infty}$} algebra,'' {\em J.
  Math. Phys.} {\bfseries 48} no.~12, (2007) 123520, 35.

\bibitem{FHHSY2009commutative}
B.~Feigin, K.~Hashizume, A.~Hoshino, J.~Shiraishi, and S.~Yanagida, ``A
  commutative algebra on degenerate {$\Bbb{CP}^1$} and {M}acdonald
  polynomials,'' {\em J. Math. Phys.} {\bfseries 50} no.~9, (2009) 095215, 42
  {\ttfamily arXiv:0904.2291 [math.CO]}.

\bibitem{SKAO1995quantum}
J.~Shiraishi, H.~Kubo, H.~Awata, and S.~Odake, ``{A Quantum deformation of the
  Virasoro algebra and the Macdonald symmetric functions},''
{\em Lett. Math. Phys.} {\bfseries 38} (1996) 33--51 {\ttfamily
  arXiv:q-alg/9507034 [q-alg]}.
%%CITATION = Q-ALG/9507034;%%.

\bibitem{AKOS1995Quantum}
H.~Awata, H.~Kubo, S.~Odake, and J.~Shiraishi, ``{Quantum $\mathscr{W}_N$
  algebras and Macdonald polynomials},''
{\em Comm. Math. Phys.} {\bfseries 179} (1996) 401--416 {\ttfamily
  arXiv:q-alg/9508011 [q-alg]}.
%%CITATION = Q-ALG/9508011;%%.

\bibitem{FF1995quantum}
B.~Feigin and E.~Frenkel, ``{Quantum {$\mathcal{W}$}-algebras and elliptic
  algebras},'' {\em Commun. Math. Phys.} {\bfseries 178} (1996) 653--678
  {\ttfamily arXiv:q-alg/9508009}.

\bibitem{FHSSY2010Kernel}
B.~Feigin, A.~Hoshino, J.~Shibahara, J.~Shiraishi, and S.~Yanagida, ``Kernel
  function and quantum algebra,'' {\em RIMS kokyuroku} {\bfseries 1689} (2010)
  133--152 {\ttfamily arXiv:1002.2485 [math.QA]}.

\bibitem{AFHKSY2011notes}
H.~Awata, B.~Feigin, A.~Hoshino, M.~Kanai, J.~Shiraishi, and S.~Yanagida,
  ``{Notes on Ding-Iohara algebra and AGT conjecture},''
{\em RIMS kokyuroku} {\bfseries 1765} (2011) 12--32 {\ttfamily arXiv:1106.4088
  [math-ph]}.
%%CITATION = ARXIV:1106.4088;%%.

\bibitem{AGT2010liouville}
L.~F. Alday, D.~Gaiotto, and Y.~Tachikawa, ``{Liouville Correlation Functions
  from Four-dimensional Gauge Theories},''
{\em Lett. Math. Phys.} {\bfseries 91} (2010) 167--197 {\ttfamily
  arXiv:0906.3219 [hep-th]}.
%%CITATION = ARXIV:0906.3219;%%.

\bibitem{Ohkubo2017Kac}
Y.~Ohkubo, ``{Kac determinant and singular vector of the level $N$
  representation of Ding-Iohara-Miki algebra},'' {\em Lett. Math. Phys.}
  {\bfseries 109} no.~1, (2019) 33--60 {\ttfamily arXiv:1706.02243 [math-ph]}.

\bibitem{FOS2020Generalized}
M.~Fukuda, Y.~Ohkubo, and J.~Shiraishi, ``{Generalized Macdonald Functions on
  Fock Tensor Spaces and Duality Formula for Changing Preferred Direction},''
  {\em Commun. Math. Phys.} {\bfseries 380} no.~1, (2020) 1--70 {\ttfamily
  arXiv:1903.05905 [math.QA]}.

\bibitem{Negut2016qAGTW}
A.~Negut, ``{The q-AGT-W relations via shuffle algebras},'' {\em Commun. Math.
  Phys.} {\bfseries 358} no.~1, (2018) 101--170 {\ttfamily arXiv:1608.08613
  [math.RT]}.

\bibitem{AFLT2011combinatorial}
V.~A. Alba, V.~A. Fateev, A.~V. Litvinov, and G.~M. Tarnopolskiy, ``{On
  combinatorial expansion of the conformal blocks arising from AGT
  conjecture},''
{\em Lett. Math. Phys.} {\bfseries 98} (2011) 33--64 {\ttfamily arXiv:1012.1312
  [hep-th]}.
%%CITATION = ARXIV:1012.1312;%%.

\bibitem{FL2011Integrable}
V.~A. Fateev and A.~V. Litvinov, ``{Integrable structure, W-symmetry and AGT
  relation},''
{\em JHEP} {\bfseries 01} (2012) 051 {\ttfamily arXiv:1109.4042 [hep-th]}.
%%CITATION = ARXIV:1109.4042;%%.

\bibitem{Ohkubo2025toward}
Y.~Ohkubo, ``Toward the {$q$}-deformation of the {$N=1$} superconformal
  algebra,'' {\em Bulletin of the Transdisciplinary and Interdisciplinary
  Science} {\bfseries 10} no.~1, (2025) 135--154 (Japanese).

\bibitem{BF2011Super}
V.~Belavin and B.~Feigin, ``{Super Liouville conformal blocks from N=2 SU(2)
  quiver gauge theories},'' {\em JHEP} {\bfseries 07} (2011) 079 {\ttfamily
  arXiv:1105.5800 [hep-th]}.

\bibitem{BBFLT2011Instanton}
A.~A. Belavin, M.~A. Bershtein, B.~L. Feigin, A.~V. Litvinov, and G.~M.
  Tarnopolsky, ``{Instanton moduli spaces and bases in coset conformal field
  theory},''
{\em Comm. Math. Phys.} {\bfseries 319} (2013) 269--301 {\ttfamily
  arXiv:1111.2803 [hep-th]}.
%%CITATION = ARXIV:1111.2803;%%.

\bibitem{BBT2012Bases}
A.~A. Belavin, M.~A. Bershtein, and G.~M. Tarnopolsky, ``{Bases in coset
  conformal field theory from AGT correspondence and Macdonald polynomials at
  the roots of unity},''
{\em JHEP} {\bfseries 03} (2013) 019 {\ttfamily arXiv:1211.2788 [hep-th]}.
%%CITATION = ARXIV:1211.2788;%%.

\bibitem{Ito2011Ramond}
Y.~Ito, ``{Ramond sector of super Liouville theory from instantons on an ALE
  space},'' {\em Nucl. Phys. B} {\bfseries 861} (2012) 387--402 {\ttfamily
  arXiv:1110.2176 [hep-th]}.

\bibitem{NT2011Central}
T.~Nishioka and Y.~Tachikawa, ``{Central charges of para-Liouville and Toda
  theories from M-5-branes},'' {\em Phys. Rev. D} {\bfseries 84} (2011) 046009
  {\ttfamily arXiv:1106.1172 [hep-th]}.

\bibitem{IOY20132d4d}
H.~Itoyama, T.~Oota, and R.~Yoshioka, ``{2d-4d Connection between
  $q$-Virasoro/W Block at Root of Unity Limit and Instanton Partition Function
  on ALE Space},''
{\em Nucl. Phys.} {\bfseries B877} (2013) 506--537 {\ttfamily arXiv:1308.2068
  [hep-th]}.
%%CITATION = ARXIV:1308.2068;%%.

\bibitem{IOY2014qVirasoro}
H.~Itoyama, T.~Oota, and R.~Yoshioka, ``{$q$-Virasoro/W Algebra at Root of
  Unity and Parafermions},''
{\em Nucl. Phys.} {\bfseries B889} (2014) 25--35 {\ttfamily arXiv:1408.4216
  [hep-th]}.
%%CITATION = ARXIV:1408.4216;%%.

\bibitem{AFS2012quantum}
H.~Awata, B.~Feigin, and J.~Shiraishi, ``{Quantum Algebraic Approach to Refined
  Topological Vertex},''
{\em JHEP} {\bfseries 03} (2012) 041 {\ttfamily arXiv:1112.6074 [hep-th]}.
%%CITATION = ARXIV:1112.6074;%%.

\bibitem{AK2008Refined}
H.~Awata and H.~Kanno, ``{Refined BPS state counting from Nekrasov's formula
  and Macdonald functions},''
{\em Int. J. Mod. Phys.} {\bfseries A24} (2009) 2253--2306 {\ttfamily
  arXiv:0805.0191 [hep-th]}.
%%CITATION = ARXIV:0805.0191;%%.

\bibitem{IKV2007Refined}
A.~Iqbal, C.~Kozcaz, and C.~Vafa, ``{The Refined topological vertex},''
{\em JHEP} {\bfseries 10} (2009) 069 {\ttfamily arXiv:hep-th/0701156 [hep-th]}.
%%CITATION = HEP-TH/0701156;%%.

\bibitem{AKMMSZ2018KZ}
H.~Awata, H.~Kanno, A.~Mironov, A.~Morozov, K.~Suetake, and Y.~Zenkevich,
  ``{$(q,t)$-KZ equations for quantum toroidal algebra and Nekrasov partition
  functions on ALE spaces},'' {\em JHEP} {\bfseries 03} (2018) 192 {\ttfamily
  arXiv:1712.08016 [hep-th]}.

\bibitem{AHKS2024quantum}
H.~Awata, K.~Harada, H.~Kanno, and J.~Shiraishi, ``A quantum deformation of the
  {$\mathcal{N}=2$} superconformal algebra,'' {\ttfamily arXiv:2407.00901
  [math.QA]}.

\bibitem{Zenkevich2019gln}
Y.~Zenkevich, ``{$ \mathfrak{gl} _{N}$ Higgsed networks},'' {\em JHEP}
  {\bfseries 12} (2021) 034 {\ttfamily arXiv:1912.13372 [hep-th]}.

\bibitem{Matsuo1994qdeformation}
A.~Matsuo, ``{A {$q$}-deformation of Wakimoto modules, primary fields and
  screening operators},'' {\em Commun. Math. Phys.} {\bfseries 160} (1994)
  33--48 {\ttfamily arXiv:hep-th/9212040 [hep-th]}.

\bibitem{KM1986null}
M.~Kato and S.~Matsuda, ``Null field construction in conformal and
  superconformal algebras,'' {\em Conformal Field Theory and Solvable Lattice
  Models. Adv. Stud. Pure Math} {\bfseries 16} (1986) 205--254.

\bibitem{BV2022NSR}
M.~Bershtein and A.~Vargulevich, ``{NSR singular vectors from Uglov
  polynomials},'' {\em J. Math. Phys.} {\bfseries 63} no.~6, (2022) 061706
  {\ttfamily arXiv:2202.11810 [math-ph]}.

\bibitem{Yanagida2015singular}
S.~Yanagida, ``Singular vectors of $ N= 1$ super Virasoro algebra via Uglov
  symmetric functions,'' {\ttfamily arXiv:1508.06036 [math.QA]}.

\bibitem{MY1995Singular}
K.~Mimachi and Y.~Yamada, ``Singular vectors of the {V}irasoro algebra in terms
  of {J}ack symmetric polynomials,'' {\em Comm. Math. Phys.} {\bfseries 174}
  no.~2, (1995) 447--455.

\bibitem{MY1995RIMSkokyuroku}
K.~Mimachi and Y.~Yamada, ``Singular vectors of Virasoro algebra in terms of
  Jack symmetric polynomials,'' {\em RIMS kokyuroku} {\bfseries 919} (1995)
  68--78.

\bibitem{AMOS1995Excited}
H.~Awata, Y.~Matsuo, S.~Odake, and J.~Shiraishi, ``Excited states of the
  {C}alogero-{S}utherland model and singular vectors of the {$W_N$} algebra,''
  {\em Nuclear Phys. B} {\bfseries 449} no.~1-2, (1995) 347--374 {\ttfamily
  arXiv:hep-th/9503043 [hep-th]}.

\bibitem{FJMM2016branching}
B.~Feigin, M.~Jimbo, T.~Miwa, and E.~Mukhin, ``Branching rules for quantum
  toroidal ${\mathfrak{gl}}_n$,'' {\em Adv. Math.} {\bfseries 300} (2016)
  229--274 {\ttfamily arXiv:1309.2147 [math.QA]}.

\bibitem{FJMV2021Deformations}
B.~Feigin, M.~Jimbo, E.~Mukhin, and I.~Vilkoviskiy, ``{Deformations of
  $\mathcal {W}$ algebras via quantum toroidal algebras},'' {\em Selecta Math.}
  {\bfseries 27} no.~4, (2021) 52 {\ttfamily arXiv:2003.04234 [math.QA]}.

\bibitem{FJM2018towards}
B.~Feigin, M.~Jimbo, and E.~Mukhin, ``{Towards trigonometric deformation of
  $\widehat{\mathfrak{sl}}_2$ coset VOA},'' {\em J. Math. Phys.} {\bfseries 60}
  no.~7, (2019) 073507 {\ttfamily arXiv:1811.02056 [math.QA]}.

\bibitem{GR2017Vertex}
D.~Gaiotto and M.~Rap\v{c}\'ak, ``{Vertex Algebras at the Corner},'' {\em JHEP}
  {\bfseries 01} (2019) 160 {\ttfamily arXiv:1703.00982 [hep-th]}.

\bibitem{Harada2020Quantum}
K.~Harada, ``{Quantum deformation of Feigin-Semikhatov's W-algebras and 5d AGT
  correspondence with a simple surface operator},'' {\ttfamily arXiv:2005.14174
  [hep-th]}.

\bibitem{HMNW2021deformation}
K.~Harada, Y.~Matsuo, G.~Noshita, and A.~Watanabe, ``{$q$-deformation of corner
  vertex operator algebras by Miura transformation},'' {\em JHEP} {\bfseries
  04} (2021) 202 {\ttfamily arXiv:2101.03953 [hep-th]}.

\bibitem{Saito1998quantum}
Y.~Saito, ``Quantum toroidal algebras and their vertex representations,'' {\em
  Publ. RIMS} {\bfseries 34} no.~2, (1998) 155--177 {\ttfamily arXiv:9611030
  [q-alg]}.

\bibitem{STU1998toroidal}
Y.~Saito, K.~Takemura, and D.~Uglov, ``Toroidal actions on level 1 modules of
  {$U_q(\widehat{\mathfrak{sl}}_n)$},'' {\em Transformation Groups} {\bfseries
  3} no.~1, (1998) 75--102 {\ttfamily arXiv:q-alg/9702024 [math.QA]}.

\bibitem{JLMP1996Lukyanov}
M.~Jimbo, M.~Lashkevich, T.~Miwa, and Y.~Pugai, ``{Lukyanov's screening
  operators for the deformed Virasoro algebra},'' {\em Phys. Lett. A}
  {\bfseries 229} (1997) 285--292 {\ttfamily arXiv:hep-th/9607177}.

\bibitem{FJM2021Evaluation}
B.~Feigin, M.~Jimbo, and E.~Mukhin, {\em Evaluation Modules for Quantum
  Toroidal {$\mathfrak{gl}_n$} Algebras (Interactions of Quantum Affine
  Algebras with Cluster Algebras, Current Algebras and Categorification)},
  pp.~393--425.
\newblock Springer International Publishing,
\newblock 2021 {\ttfamily arXiv:1709.01592 [math.QA]}.

\bibitem{Wyllard2011Coset}
N.~Wyllard, ``{Coset conformal blocks and N=2 gauge theories},'' {\ttfamily
  arXiv:1109.4264 [hep-th]}.

\bibitem{CSS1989Renormalization}
C.~Crnkovic, G.~M. Sotkov, and M.~Stanishkov, ``{Renormalization Group Flow for
  General SU(2) Coset Models},'' {\em Phys. Lett. B} {\bfseries 226} (1989)
  297--301.

\bibitem{CPSS1989Fusions}
C.~Crnkovic, R.~Paunov, G.~M. Sotkov, and M.~Stanishkov, ``{Fusions of
  Conformal Models},'' {\em Nucl. Phys. B} {\bfseries 336} (1990) 637--690.

\bibitem{Lashkevich1992Superconf}
M.~Y. Lashkevich, ``{Superconformal 2-D minimal models and an unusual coset
  construction},'' {\em Mod. Phys. Lett. A} {\bfseries 8} (1993) 851--860
  {\ttfamily arXiv:hep-th/9301093}.

\end{thebibliography}
%%%%%%%%%%%%%%%%%%%%%%%%%%%%%%%%%%%%%
%%%%%%%%%%%%%%Reference%%%%%%%%%%%%%%
%%%%%%%%%%%%%%%%%%%%%%%%%%%%%%%%%%%%%

\end{document}